\documentclass[reqno,10pt]{amsart}

\usepackage[a4paper,left=25mm,right=25mm,top=30mm,bottom=30mm,marginpar=25mm]{geometry}

\usepackage{amsmath}
\usepackage{amssymb}
\usepackage{amsthm}
\usepackage{amscd}
\usepackage[ansinew]{inputenc}
\usepackage{color}
\usepackage[final]{graphicx}
\usepackage{tikz}
\usepackage{bbm}
\usepackage{comment}
\usepackage{pgfplots}
\usetikzlibrary{intersections, pgfplots.fillbetween}

\usetikzlibrary{patterns} 
\usetikzlibrary{positioning}
\usetikzlibrary{calc}

\usetikzlibrary{intersections}

\usepackage{float}
 \usepackage[colorlinks=true, pdfstartview=FitV, 
 linkcolor=olive
 ,             citecolor=olive
 , urlcolor=olive
 ]{hyperref}

\renewcommand{\epsilon}{\varepsilon}

\numberwithin{equation}{section}

\newtheoremstyle{thmlemcorr}{10pt}{10pt}{\itshape}{}{\bfseries}{.}{10pt}{{\thmname{#1}\thmnumber{
#2}\thmnote{ (#3)}}}
\newtheoremstyle{thmlemcorr*}{10pt}{10pt}{\itshape}{}{\bfseries}{.}\newline{{\thmname{#1}\thmnumber{
#2}\thmnote{ (#3)}}}
\newtheoremstyle{defi}{10pt}{10pt}{\itshape}{}{\bfseries}{.}{10pt}{{\thmname{#1}\thmnumber{
#2}\thmnote{ (#3)}}}
\newtheoremstyle{remexample}{10pt}{10pt}{}{}{\bfseries}{.}{10pt}{{\thmname{#1}\thmnumber{
#2}\thmnote{ (#3)}}}
\newtheoremstyle{ass}{10pt}{10pt}{}{}{\bfseries}{.}{10pt}{{\thmname{#1}\thmnumber{
A#2}\thmnote{ (#3)}}}

\theoremstyle{thmlemcorr}
\newtheorem{theorem}{Theorem}
\numberwithin{theorem}{section}
\newtheorem{lemma}[theorem]{Lemma}
\newtheorem{corollary}[theorem]{Corollary}
\newtheorem{proposition}[theorem]{Proposition}

\theoremstyle{thmlemcorr*}
\newtheorem{theorem*}{Theorem}
\newtheorem{lemma*}[theorem]{Lemma}
\newtheorem{corollary*}[theorem]{Corollary}
\newtheorem{proposition*}[theorem]{Proposition}
\newtheorem{problem*}[theorem]{Problem}
\newtheorem{conjecture*}[theorem]{Conjecture}

\theoremstyle{defi}
\newtheorem{definition}[theorem]{Definition}

\theoremstyle{remexample}
\newtheorem{remark}[theorem]{Remark}
\newtheorem{example}[theorem]{Example}

\theoremstyle{ass}

\newcommand{\Ccal}{\mathcal{C}}

\newcommand{\Fcal}{\mathcal{F}}
\newcommand{\Gcal}{\mathcal{G}}

\newcommand{\Ical}{\mathcal{I}}
\newcommand{\Jcal}{\mathcal{J}}

\newcommand{\Rcal}{\mathcal{R}}

\newcommand{\Tcal}{\mathcal{T}}

\newcommand{\Hbb}{\mathbb{H}}

\newcommand{\NN}{\mathbb{N}}

\newcommand{\Sbb}{\mathbb{S}}

\DeclareMathOperator*{\argmin}{arg\,min}
\DeclareMathOperator{\esssup}{ess\,sup}
\DeclareMathOperator{\essinf}{ess\,inf}

\DeclareMathOperator{\supp}{supp}

\newcommand{\norm}[1]{\|#1\|}
\newcommand{\normlr}[1]{\left\|#1\right\|}

\newcommand{\normB}[1]{\Bigl\|#1\Bigr\|}
\newcommand{\normBB}[1]{\biggl\|#1\biggr\|}

\newcommand{\abs}[1]{|#1|}

\newcommand{\absB}[1]{\Bigl|#1\Bigr|}
\newcommand{\absBB}[1]{\biggl|#1\biggr|}

\newcommand{\dprlr}[1]{\left\langle #1 \right\rangle}

\newcommand{\dd}{\;\mathrm{d}}

\newcommand{\N}{\mathbb{N}}
\newcommand{\R}{\mathbb{R}}
\newcommand{\RR}{\mathbb{R}}

\newcommand{\Z}{\mathbb{Z}}

\newcommand{\weakly}{\rightharpoonup}

\newcommand{\ffi}{\varphi}

\allowdisplaybreaks

\def\dd{{\rm d}}
\def\dx{{\rm d}x}
\def\dy{{\rm d}y}

\def\leq{\leqslant}
\def\geq{\geqslant}

\newcommand{\ucl}{u^{c}}
\newcommand{\uet}{u^{\eta}}

\newcommand{\ubet}{w^{(\beta)}}

 \definecolor{Korange}{rgb}{0.945,0.561,0}
 
\usepackage{esint}
 \newcommand{\dashint}{\fint}

\usepackage[nocompress, space]{cite}
 
\title[Structural changes through bi-level parameter learning]{Structural
changes in nonlocal denoising models arising through bi-level parameter
learning}

\author[Elisa Davoli]{Elisa Davoli}
\address{TU Wien, Institute of Analysis and Scientific Computing, Wiedner Hauptstrasse 8-10, 1040 Vienna, Austria}
\email{elisa.davoli@tuwien.ac.at}

\author[Rita Ferreira]{Rita Ferreira}
\address{King Abdullah University of Science and Technology (KAUST), CEMSE Division, Thuwal 23955-6900, Saudi Arabia}
\email{rita.ferreira@kaust.edu.sa}

\author[Carolin Kreisbeck]{Carolin Kreisbeck}
\address{Mathematisch-Geographische Fakult\"at, KU Eichst\"att-Ingolstadt, Ostenstrasse 26, 85072, Eichst\"att, Germany}
\email{carolin.kreisbeck@ku.de}

\author[Hidde Sch\"onberger]{Hidde Sch\"onberger}
\address{Mathematisch-Geographische Fakult\"at, KU Eichst\"att-Ingolstadt, Ostenstrasse 26, 85072, Eichst\"att, Germany}
\email{hidde.schoenberger@ku.de}

\usepackage[normalem]{ulem}

\begin{document}

 
\maketitle

 \begin{abstract}  
\vspace{-12pt} 

We introduce a unified framework based on bi-level optimization schemes to deal with parameter learning in the context of image processing. The goal is to identify the optimal regularizer within a family depending on a parameter in a general topological space. Our focus lies on the situation with non-compact parameter domains, which is, for example, relevant when the commonly used box constraints are disposed of. To overcome this lack of compactness, we propose a natural extension of the upper-level functional to the closure of the parameter domain via Gamma-convergence, which captures possible structural changes in the reconstruction model at the edge of the domain. Under two main assumptions, namely, Mosco-convergence of the regularizers and uniqueness of minimizers of the lower-level problem, we prove that the extension coincides with the relaxation, thus admitting minimizers that relate to the parameter optimization problem of interest. 
We apply our abstract framework to investigate a quartet of practically relevant models in image denoising, all featuring nonlocality. The associated families of regularizers exhibit qualitatively different parameter dependence, describing a weight factor, an amount of nonlocality, an integrability exponent, and a fractional order, respectively. After the asymptotic analysis that determines the relaxation in each of the four settings, we finally establish theoretical conditions on the data that guarantee structural stability of the models and give examples of when stability is lost.

\vspace{8pt}

 
 
 \noindent\textsc{Date:} \today.
 \end{abstract}

 \section{Introduction}

\thispagestyle{empty}

One of the most widely used methods to solve image restoration problems is the variational regularization approach. This variational approach  consists of minimizing a reconstruction functional
that decomposes into a fidelity and a regularization terms, which give rise to competing effects. While the fidelity term ensures that the reconstructed image is close to the (noisy) data, the regularization term is designed to remove the noise by incorporating prior information on the clean image. 
In the case of a simple $L^2$-fidelity term, the reconstruction functional is given by
\[
\Jcal(u) = \norm{u-u^{\eta}}^2_{L^2(\Omega)} + \Rcal(u), \qquad \text{for $u \in L^2(\Omega)$},
\]
where $\Omega \subset \R^n$ is the image domain, $u^{\eta} \in L^2(\Omega)$ the noisy image, and $\Rcal:L^2(\Omega) \to [0,\infty]$ the regularizer.

A common choice for $\Rcal$ is the total variation ($TV$) regularization proposed by Rudin, Osher,  \& Fatemi \cite{ROF92}, which penalizes sharp oscillations, but does not exclude edge discontinuities, as they appear in most images. Since its introduction, the $TV$-model has inspired a variety of more advanced regularization terms,
like the infimal-convolution total variation ($ICTV$) \cite{ChL97}, the total generalized variation ($TGV$) \cite{BKP10}, and many more, cf.~\cite{BeB18} and the references therein. Due to the versatility of the variational formulation, regularizers of a completely different type can be used as well. Recently, a lot of attention has been directed towards regularizers incorporating nonlocal effects, such as those~induced by difference quotients \cite{AuKo09, GiOs08, BrNg18, BEPS11} and fractional operators \cite{Antiletal22, AnB17, AnR19}. Nonlocal regularizers have the advantage  of not requiring the existence of (full) derivatives, allowing to work with functions that are less regular than those in the local counterpart.

With an abundance of available choices, finding a suitable regularization term for a specific application is paramount for obtaining accurate reconstructions. This is often done by fixing a parameter-dependent family of regularizers and tuning the parameter in accordance with the noise and data. Carrying out this process via trial and error
can be hard and inefficient, which led to the development of a more structured approach in the form of bi-level optimization. We refer, e.g., to  \cite{DeScVa16,DeScVa17}
(see also \cite{ChPoRaBi13,ChRaPo14,Do12,TaLiAdFr07}) and to the references therein, as well as to \cite{DeZ20} for a detailed overview. 
The idea behind bi-level optimization is to employ a supervised learning scheme based on a set of training data consisting of noisy images and their corresponding clean versions. To determine an optimal parameter, one minimizes a selected cost functional which quantifies the error with respect to the training data. Overall, this results in a nested variational problem with upper- and lower-level optimization steps related to the cost and reconstruction functional, respectively.
Key aspects of the mathematical study of these 
bi-level learning schemes include establishing the existence of solutions and deriving optimality conditions, which lay the foundation for devising reliable numerical solution methods.

In recent years, there has been a rapid growth in the literature devoted to addressing the above questions.
To mention but a few examples, we first refer the paper \cite{HKB18} dealing with learning real-valued weight parameters in front of the regularization terms for a rather general class of inverse problems;
in \cite{ADK20,BaW20}, the authors   optimize the fractional parameter of a regularizer depending on the spectral fractional Laplacian;
spatially dependent weights are determined through training via other nonlocal bi-level schemes (e.g.,~inside the Gagliardo semi-norm \cite{HoK22} or in a type of fractional gradient \cite{DDM21}), and in classical $TV$-models \cite{ChDeSc17, HiR17, PPR21}; as done in \cite{DeSc13}, one can also learn the fidelity term instead of the regularizer.

A common denominator in the above references is the presence of certain a priori compactness constraints on the set of admissible parameters, such as box constraints like in \cite{HKB18}, where the weights are assumed to lie in some compact interval away from 0 and infinity. These conditions make it possible to prove stability of the lower-level problem and obtain existence of optimal parameters within a class of structurally equivalent regularizers. However, imposing artificial restrictions to the parameter range like these may lead to suboptimal results depending on the given training data.

It is then  substantial
to consider removing such constraints in order to work on maximal domains naturally associated with the parameters, which is also our focus in this  paper. An inherent effect of this approach is that qualitative changes in the structure of the regularizer may occur at the edges of the domain. If optimal parameters are attained at the boundary, this indicates that the chosen class of regularization terms is not well-suited to the training data. To exclude these degenerate cases, it is of interest to provide analytic conditions to guarantee that the optimal parameters are attained in the interior of the domain, thereby preserving the structure of the regularizer.
The first work to address the aforementioned tasks is \cite{DeScVa16} by De Los Reyes, Sch\"{o}nlieb,  \&  Valkonen, where optimization is carried out for  weighted sums of local regularizers of different type with each weight factor allowed to take any value in $[0,\infty]$.  
As such, their bi-level scheme is able to encompass multiple regularization structures at once, like $TV$ and $TV^2$ and their interpolation $TGV$. 
Similarly, the authors in \cite{LiS19}  vary the weight factor in the whole range $[0,\infty]$ as well as the underlying finite-dimensional norm of the total variation regularizer.
We also mention \cite{DaL18}, where the order of a newly introduced nonlocal counterpart of the $TGV$-regularizer is tuned, and \cite{DFL19}, which studies a bi-level scheme covering the cases of $TV$, $TGV^2$, and $NsTGV^2$ in a comprehensive way.\medskip 

In this paper, we introduce a unified framework to deal with parameter learning beyond structural stability in the context of bi-level optimization schemes.
In contrast to the above references, where the analysis is tailored to a specifically chosen type of parameter dependence, our regularizers can exhibit a general dependence on parameters in a topological space.
Precisely, we consider a parametrized family of regularizers $\Rcal_{\lambda}:L^2(\Omega) \to [0,\infty]$ with $\lambda$ ranging over a subset $\Lambda$ of a topological space $X$, which is assumed to be first countable. If we focus for brevity on a single data point $(u^c,u^\eta) \in L^2(\Omega)\times L^2(\Omega)$, with $u^c$ and $u^\eta$ the clean and noisy images (see Section~\ref{sec:general} for larger data sets), 
the bi-level optimization problem reads:
\begin{align*}
\begin{split}
({\rm Upper\text{-}level}) \qquad& \text{Minimize} \ \ \Ical(\lambda):=\inf_{w \in K_{\lambda}}\norm{w-u^c}^2_{L^2(\Omega)} \ \ \text{over $\lambda \in \Lambda$,}\\
({\rm Lower\text{-}level}) \qquad& K_{\lambda}:=\argmin_{u \in L^2(\Omega)}\Jcal_{\lambda}(u)
,\qquad
\end{split}
\end{align*}
where $\Jcal_{\lambda}(u):=\norm{u-u^{\eta}}^2_{L^2(\Omega)} + \Rcal_{\lambda}(u)$ is the reconstruction functional.

Our approach for studying this general bi-level learning scheme relies on asymptotic tools from the calculus of variations. We define a suitable notion of stability for the lower-level problems that requires the family of functionals $\{\Jcal_{\lambda}\}_{\lambda \in \Lambda}$ to be closed under taking $\Gamma$-limits; see \cite{Dal93, Bra02} for a comprehensive introduction on $\Gamma$-convergence. Since $\Gamma$-convergence ensures the convergence of sequences of minimizers, one can conclude that, in the presence of stability, the upper-level functional $\Ical$ admits a minimizer (Theorem~\ref{theo:suff}).  

A different strategy is required to obtain the existence of solutions when stability fails. Especially relevant here is the case of real-valued parameters when box constraints are disposed of and 
 non-closed intervals $\Lambda$ are considered; clearly, stability is then lost for the simple fact that a sequence of parameters can converge to the boundary of $\Lambda$. To overcome this issue, we propose a natural extension $\overline{\Ical}:\overline{\Lambda} \to [0,\infty]$ of $\Ical$, now defined on the closure of our parameter domain, and identified via $\Gamma$-convergence of the lower-level functionals. Precisely, 
\begin{align*}
\begin{split}
({\rm Upper\text{-}level}) \qquad& \text{Minimize} \ \ \overline{\Ical}(\lambda):=\inf_{w \in \overline{K}_\lambda}\norm{w-u^c}^2_{L^2(\Omega)} \ \ \text{over $\lambda \in \overline{\Lambda}$,}\qquad \ \ \\
({\rm Lower\text{-}level}) \qquad& \overline{K}_\lambda:= \argmin_{u \in L^2(\Omega)}\overline{\Jcal}_{\lambda}(u),
\end{split}
\end{align*}
where the functionals $\overline{\Jcal}_\lambda :L^2(\Omega)\to [0,\infty]$ are characterized as $L^2$-weak $\Gamma$-limits (if they exist) of functionals $\Jcal_{\lambda'}$ with $\lambda' \to \lambda$. To justify the choice of this particular extension, we derive an intrinsic connection with relaxation theory in the calculus of variations (for an introduction, see, e.g.,~\cite[Chapter~9]{Dac08} and the references therein). Explicitly, the relaxation of the upper-level functional $\Ical$ is given by its lower semicontinuous envelope (after the trivial extension to $\overline{\Lambda}$ by $\infty$), 
\[
\Ical^{\rm rlx}(\lambda):=\inf\Big\{\liminf_{k\to \infty} \Ical (\lambda_k): (\lambda_k)_k\subset \Lambda, \lambda_k\to \lambda \text{ in $\overline{\Lambda}$}\Big\} \enspace \text{ for \(\lambda\in\overline\Lambda\)}.
\]
This relaxed version of $\Ical$ has the desirable property that it admits a minimizer (if $\overline{\Lambda}$ is compact) and minimizing sequences of $\Ical$ have subsequences that converge to an optimal parameter of $\Ical^{\rm rlx}$. Our main theoretical result (Theorem~\ref{theo:relax2}) shows that the extension $\overline{\Ical}$ coincides with the relaxation $\Ical^{\rm rlx}$ under suitable assumptions and therefore inherits the same properties (cf.~Corollary~\ref{cor:relax}). 

Besides the generic conditions that each $\Rcal_{\lambda}$ is weakly lower semicontinuous and has non-empty domain (see~\eqref{eq:assumptions}), which ensure that $\Jcal_{\lambda}$ possesses a minimizer, we work under two main assumptions:
\begin{itemize}
\item[$(i)$] The Mosco-convergence of the regularizers, i.e., $\Gamma$-convergence with respect to the strong and weak $L^2$-topology, and

\item[$(ii)$]  the uniqueness of minimizers of $\overline{\Jcal}_{\lambda}$ for $\lambda \in \overline{\Lambda} \setminus \Lambda$.
\end{itemize} 
We demonstrate in Example~\ref{ex:optimal} that these assumptions are in fact optimal. Due to $(i)$, the $\Gamma$-limits $\overline{\Jcal}_{\lambda}$ preserve the additive decomposition into the $L^2$-fidelity term and a regularizer, 
and coincide with $\Jcal_{\lambda}$ inside $\Lambda$. As a consequence of the latter, it follows that $\overline{\Ical}=\Ical$ in $\Lambda$, making $\overline{\Ical}$ a true extension of $\Ical$. For the parameter values at the boundary, $\lambda \in \overline{\Lambda}\setminus \Lambda$, however, the regularizers present in $\overline{\Jcal}_{\lambda}$ can have a completely different structure from the family of regularizers $\{\Rcal_{\lambda}\}_{\lambda \in \Lambda}$ that we initially started with.  When the optimal parameter of the extended problem is attained inside $\Lambda$, one recovers instead a solution to the original training scheme, yielding structure preservation. For a discussion on related results in the context of optimal control problems \cite{BuD82,But87,BBF93}, we refer to the end of Section~\ref{sec:general}. \medskip

To demonstrate the applicability of our abstract framework, we investigate 
a quartet of practically relevant scenarios 
with 
families of nonlocal regularizers that induce qualitatively different structural changes; namely, learning the optimal weight, varying the amount of nonlocality, optimizing the integrability exponent, and tuning the fractional parameter. More precisely, in all these four applications, our starting point is a non-closed real interval $\Lambda\subset [-\infty,\infty]$ and we seek to determine the extension $\overline{\Ical}$ on the closed interval $\overline{\Lambda}$, which admits a minimizer by the theory outlined above. The first step is to calculate the  Mosco-limits of the regularizers, which reveals the type of structural change occurring at the boundary points. \color{black} Subsequently, we study for which training sets of clean and noisy images the optimal parameters are attained either inside $\Lambda$ or at the edges. In two cases, we determine explicit analytic conditions on the data that guarantee structure preservation for the optimization process.

The first setting 
involves a rather general nonlocal regularizer $\Rcal:L^2(\Omega) \to [0,\infty]$ multiplied by a weight parameter $\alpha$ in $\Lambda = (0,\infty)$. 
Inside the domain, we observe structural stability as $\overline{\Jcal}_{\alpha}=\Jcal_{\alpha}$ for all $\alpha \in \Lambda$; in contrast, the regularization disappears when $\alpha =0$ and forces the solutions to be constant when $\alpha=\infty$. 
Moreover, we derive sufficient conditions in terms of the data that prevent the optimal parameter from being attained at  the boundary points; for a single data point $(u^c,u^\eta)$, they specify to
\[
\Rcal(u^c) < \Rcal(u^\eta) \quad \text{and} \quad \norm{u^{\eta}-u^{c}}_{L^2(\Omega)}^2< \normlr{\dashint_{\Omega} u^{\eta}\,\dx - u^{c}}_{L^2(\Omega)}^2,
\]
see Theorem~\ref{thm:pos}. 
Notice that the first of these two conditions is comparable to the one in \cite[Eq.~(10)]{DeScVa16} and shows positivity of optimal weights.

Inspired by the use of different $L^p$-norms in image processing, such as~in the form of quadratic, $TV$, and Lipschitz regularization \cite[Section~4]{PCBC10}, we focus our second case on the integrability exponent of nonlocal regularizers of double-integral type; precisely, functionals of the form
\[
\Rcal_{p}(u) = \left(\frac{1}{|\Omega\times \Omega|}\int_\Omega\int_\Omega f^p(x,y, u(x), u(y)) \,\dx\,\dy \right)^{1/p} \quad \text{for $p \in \Lambda=[1,\infty)$},
\]
with a suitable $f:\Omega\times\Omega\times\R\times\R \to [0,\infty)$. Possible choices for the integrand $f$ include bounded functions or functions of difference-quotient type. We prove stability of the lower-level problem in $\Lambda$, 
and determine the Mosco-limit for $p \to \infty$ via $L^p$-approximation techniques as in \cite{ChPaPr04, KRZ22}. In particular, we show that it is given by a double-supremal functional of the form
\[
\mathcal R_\infty (u)= \esssup_{(x,y)\in \Omega\times \Omega} f(x,y, u(x), u(y)).
\]
In order to see how this structural change affects the image reconstruction, we highlight examples of training data for which the supremal regularizer performs better or worse than the integral counterparts.

As a third application, we consider two families of nonlocal regularizers $\{\Rcal_{\delta}\}_{\delta \in \Lambda}$ with $\Lambda = (0,\infty)$, which were introduced by Aubert \& Kornprobst \cite{AuKo09} and Brezis \& Nguyen in \cite{BrNg18}, respectively, and are closely related to nonlocal filters frequently used in image processing. The parameter $\delta$ reflects the amount of nonlocality in the regularizer. It is known that the functionals $\Rcal_{\delta}$ tend, as $\delta \to 0$, to a multiple of the total variation in the sense of $\Gamma$-convergence. Based on these results, we prove in both cases that the reconstruction functional of our bi-level optimization scheme turns into the classical $TV$-denoising model when $\delta =0$, whereas the regularization vanishes at the other boundary value, $\delta = \infty$. As such, the extended bi-level schemes encode simultaneously nonlocal and total variation regularizations.   
We round off the discussion by presenting some instances of training data where the optimal parameters are attained either at the boundary or in the interior of $\Lambda$. 

Our final bi-level optimization problem features a different type of nonlocality arising from fractional operators; to be precise, we consider, in the same spirit as in \cite{AnB17}, the $L^2$-norm of the spectral fractional Laplacian as a regularizer. 
The parameter of interest here is the order $s/2$ of the fractional Laplacian, which is taken in the fractional range $s \in \Lambda = (0,1)$. 
At the values $s=0$ and $s=1$, we recover local models with regularizers equal to the $L^2$-norm of the function and its gradient, respectively.  Thus, one expects the fractional model to perform better than the two local extremes. We quantify this presumption by deriving analytic conditions in terms of the eigenfunctions and eigenvalues of the classical Laplacian on $\Omega$  ensuring the optimal parameters to be attained in the truly fractional regime.  These conditions on the training data are established by proving and exploiting the differentiability of the extended upper-level functional $\overline{\Ical}$.

For completeness, we mention that practically relevant scenarios when $\Lambda$ is a topological space include  those in which the reconstruction parameters are space-dependent, and thus described by functions. The analysis of this class of applications is left open for future investigations.

The outline of the paper is as follows. In Section~\ref{sec:general}, we present the general abstract bi-level framework, and prove the results regarding the existence of optimal parameters and the two types of extensions of bi-level optimization schemes. The Sections~\ref{sec:weight}--\ref{sec: fractional} then deal with the four applications mentioned in the previous paragraphs.

\section{Establishing the unified framework}\label{sec:general}

Let $\Omega\subset \R^n$ be an open bounded set, and let
\begin{align*}
\bigcup_{j=1}^N (u_j^{c}, u_j^{\eta})\subset L^2(\Omega)\times L^2(\Omega),\quad N\in \N,
\end{align*}
be a set of available square-integrable training data, where each  $u_j^{c}$ represents a clean image and $u_j^{\eta}$ a distorted version thereof, which can be obtained, for instance, by applying some noise to $u_j^{c}$. These data are collected in the vector-valued functions $u^{c} := (u_1^{c}, \ldots, u_{N}^{c})\in L^2(\Omega;\R^N)$ and $u^{\eta} := (u_1^{\eta}, \ldots, u_{N}^{\eta})\in L^2(\Omega;\R^N)$.  As for notation, $\norm{v}_{L^2(\Omega;\R^N)}^2 = \sum_{j=1}^N \norm{v_j}_{L^2(\Omega)}^2$ stands for the \(L^{2}\)-norm of  a function  $v\in L^2(\Omega;\R^N)$.

To reconstruct each damaged image,  $u_j^{\eta}$, we consider denoising  models that consist of a simple fidelity term and a (possibly nonlocal) regularizer; precisely, we minimize functionals  $\Jcal_{\lambda, j}: L^2(\Omega)\to [0, \infty]$ of the form
\begin{align}\label{general_model}
\Jcal_{\lambda, j}(u) =  \|u-u_j^{\eta}\|^2_{L^2(\Omega)}+ \Rcal_\lambda(u), \qquad u\in L^2(\Omega),
\end{align}
where the regularizer $\Rcal_\lambda:L^2(\Omega)\to [0, \infty]$, with $\text{Dom\,} \Rcal_\lambda= \{v\in L^2(\Omega): \Rcal_\lambda(u)<\infty\}$, is a (possibly nonlocal) functional parametrized over $\lambda\in \Lambda $ with $\Lambda$ a subset of a topological space $X$  satisfying the first axiom of countability. Throughout the paper, we always assume that for every $\lambda\in \Lambda$, we have
\begin{equation}\label{eq:assumptions}\tag{H}
\begin{cases}
\text{Dom\,} \Rcal_\lambda \ \text{is non-empty},\\
\Rcal_\lambda \ \text{is weakly $L^2$-lower semicontinuous.}
\end{cases}
\end{equation}
Observe that the functionals $\Jcal_{\lambda,j}$
then  have a minimizer by the direct method in the calculus of variations.

\color{black}
The result of the reconstruction process, meaning the quality of the reconstructed image resulting as a minimizer of~\eqref{general_model}, is known to depend on the choice of the regularizing term $\Rcal_\lambda$.
Our goal is to set up a training scheme that is able to learn how to select a ``good'' parameter $\lambda$ within a corresponding
 given family $\{\Rcal_\lambda\}_{\lambda\in
\Lambda}$ of regularizers.
Here, as briefly described in the Introduction for the single data point case (\(N=1)\), we follow the approach introduced in  \cite{DeScVa16,DeScVa17}
 in the spirit of machine
learning optimization schemes, where training the regularization term means to solve the nested variational problem 
\begin{align}\label{training}\tag{$\Tcal$}
\begin{split}
({\rm Upper\text{-}level}) \qquad& \text{Minimize} \ \ {\Ical}(\lambda):=\inf_{w \in K_{\lambda}}\norm{w-u^c}^2_{L^2(\Omega;\R^N)} \ \ \text{over $\lambda \in \Lambda$,} \\
({\rm Lower\text{-}level}) \qquad& K_{\lambda} : = \Bigl\{w\in L^2(\Omega;\R^N): w_j\in \argmin_{u\in L^2(\Omega)} \Jcal_{\lambda,
j}(u) \hbox{ for all } j\in\{1, \ldots, N\}\Bigr\},
\end{split}
\end{align}
with $\Jcal_{\lambda,j}$ as in~\eqref{general_model}.
 Notice that $K_\lambda\neq\emptyset$ because for all $j\in\{1, \ldots, N\}$, we have
\begin{align}\label{eq:argminK}
K_{\lambda, j} : = {\rm argmin}_{u\in L^2(\Omega)} \Jcal_{\lambda,
j}(u)\not=\emptyset
\end{align}
by Assumption~\eqref{eq:assumptions}.

  To study the training scheme \eqref{training}, we start by introducing a notion of weak \(L^2\)-stability for the family  \(\{\Jcal_{\lambda}\}_{\lambda\in \Lambda}\), with
 \begin{align}\label{eq:famfunct}
 \Jcal_\lambda:=(\Jcal_{\lambda, 1}, \ldots, \Jcal_{\lambda,
N}):L^2(\Omega)\to
[0, \infty]^N\enspace \hbox{ for \(\lambda\in\Lambda\).}
\end{align}
This notion relies on the concept of $\Gamma$-convergence and is related to the notion of (weak) stability as in~\cite[Definition~2.3]{HKB18}, which is defined in terms of minimizers of the lower-level problem.

\begin{definition}[Weak \boldmath{$L^2$}-stability]\label{def:stability}
The family in \eqref{eq:famfunct} is called weakly $L^2$-stable if for every sequence $(\lambda_k)_k\subset\Lambda$ such that  $(\Jcal_{\lambda_k, j})_k$ $\Gamma$-converges with respect to the weak $L^2$-topology for all $j\in\{1, \ldots, N\}$, there exists $\lambda\in \Lambda$ such that  
\begin{align*}
\Gamma(w\text{-}L^2)\text{-}\lim_{k\to \infty} \Jcal_{\lambda_k, j} = \Jcal_{\lambda, j}
\end{align*}
for all $j\in\{1, \ldots, N\}$.
\end{definition}

Before  proceeding, we briefly recall the definition and some properties of $\Gamma$-convergence in the setting relevant to us; for more on this topic, see~\cite{Dal93,Bra02} for instance.

\begin{definition}[\boldmath{$\Gamma$}- and Mosco-convergence]
Let $\Fcal_k:L^2(\Omega) \to [0,\infty]$ for $k \in \N$ and $\Fcal:L^2(\Omega) \to [0,\infty]$ be functionals. The sequence $(\Fcal_k)_k$  (sequentially) $\Gamma$-converges to $\Fcal$ with respect to the weak $L^2$-topology,  written $\Fcal=\Gamma(w\text{-}L^2)\text{-}\lim_{k \to \infty} \Fcal_k$, if:
\begin{itemize}
\item (Liminf inequality) For every sequence $(u_k)_k\subset L^2(\Omega)$ and $u\in L^2(\Omega)$ with $u_k \rightharpoonup u$ in $L^2(\Omega)$, it holds that
\[
\Fcal(u) \leq \liminf_{k \to \infty} \Fcal_k(u_k).
\]
\item (Limsup inequality) For every $u\in L^2(\Omega)$, there exists a sequence $(u_k)_k\subset L^2(\Omega)$ such that $u_k \rightharpoonup u$ in $L^2(\Omega)$ and
\[
\Fcal(u) \geq \limsup_{k \to \infty} \Fcal_k(u_k). \qedhere
\]
\end{itemize}
The sequence $(\Fcal_k)_k$ converges in the sense of Mosco-convergence in $L^2(\Omega)$ to $\Fcal$,  written $\Fcal=\text{\rm Mosc}(L^2)$-$\lim_{k\to \infty}\Fcal_k$, if, in addition, the limsup inequality can be realised by a sequence converging strongly in $L^2(\Omega)$.
\end{definition}
If the liminf inequality holds, then the sequence from the limsup inequality automatically satisfies $\lim_{k \to \infty} \Fcal_k(u_k) = \Fcal(u)$, and is therefore often called a recovery sequence. 
We note that the above sequential definition of $\Gamma$-convergence coincides with the topological definition \cite[Proposition~8.10]{Dal93}  for equi-coercive sequences $(\Fcal_k)_k$, i.e., $\Fcal_k \geq \Psi$ for all $k \in \N$ and for some $\Psi:L^2(\Omega) \to [0,\infty]$ with $\Psi(u) \to \infty$ as $\norm{u}_{L^2(\Omega)}\to \infty $. In particular, the theory implies that the $\Gamma$-limit $\Fcal$ is (sequentially) $L^2$-weakly lower semicontinuous. The   $\Gamma$-convergence has the key property of yielding the convergence of solutions (if they exist) to those of the limit problem, which makes it a suitable notion of variational convergence. Precisely, if $u_k$ is a minimizer of $\Fcal_{k}$ for all $k \in \N$ and $u$ a cluster point of the sequence $(u_k)_k$, then $u$ is a minimizer of $\Fcal$ and $\min_{L^2(\Omega)} \Fcal_k = \Fcal_k(u_k) \to \Fcal(u) = \min_{L^2(\Omega)} \Fcal$, see~\cite[Corollary~7.20]{Dal93}. Notice that the existence of cluster points is implied by the assumption of equi-coercivity.
 In the special case when $(\Fcal_k)_k$ is a constant sequence of functionals, say $\Fcal_k=\Gcal$ for all $k\in \N$,  the $\Gamma$-limit corresponds to the relaxation of $\Gcal$, i.e.,~its $L^2$-weakly lower semicontinuous envelope. Observe that replacing each $\Fcal_k$ by its relaxation does not affect the $\Gamma$-limit of $(\Fcal_k)_k$, see \cite[Proposition~6.11]{Dal93}. 
 
 \color{black}
As we discuss next, weak \(L^2\)-stability provides existence of solutions to the  training scheme
\eqref{training}. We note that the family of functionals $\{\Jcal_\lambda\}_{\lambda \in\Lambda}$ as in~\eqref{eq:famfunct} is equi-coercive in a componentwise sense.

\begin{theorem}
\label{theo:suff}
Let $\Jcal_{\lambda}:L^2(\Omega)\to [0, \infty]^N$ be given by \eqref{eq:famfunct} for each $\lambda\in \Lambda$. If the family $\{\Jcal_\lambda\}_{\lambda\in \Lambda}$ is weakly $L^2$-stable, then $\Ical$ in \eqref{training} has a minimizer.
\end{theorem}

\begin{proof} The statement follows directly from the direct method and the classical properties of $\Gamma$-convergence. 

Let $(\lambda_k)_k\subset \Lambda$ be a minimizing sequence for $\Ical$. Then, 
for each \(k\in\N\),  there is $w_k \in K_{\lambda_k}$ such that
\begin{align}\label{minimizingsequence}
\lim_{k\to\infty}\norm{w_k-u^{c}}^2_{L^2(\Omega;\R^N)} = \inf_{\lambda\in \Lambda} \Ical(\lambda). 
\end{align}
In particular, $(w_k)_k$ is uniformly bounded in $L^2(\Omega;\R^N)$;  hence, extracting a  subsequence if necessary, one may assume that  $w_k\weakly w$ in $L^2(\Omega;\R^N)$ as $k\to \infty$ for some $w\in L^2(\Omega;\R^N)$. 
 Using the equi-coercivity, we apply the compactness result for $\Gamma$-limits \cite[Corollary~8.12]{Dal93} to find a further  subsequence of $(\lambda_k)_k$ (not relabeled) such that $(\Jcal_{\lambda_k, j})_k$ $\Gamma(w\text{-}L^2)$-converges for all \(j\in\{1,...,N\}\).
Consequently, by the weak $L^2$-stability assumption and the properties of \(\Gamma\)-convergence on minimizing sequences,  there exists  $\tilde\lambda\in \Lambda$ such that $w \in K_{\tilde \lambda}$. 
Then, along with~\eqref{minimizingsequence},  
\begin{align*}
\Ical(\tilde \lambda) \leq \norm{w-u^{c}}^2_{L^2(\Omega;\R^N)}\leq \liminf_{k\to \infty} \norm{w_k-u^{c}}_{L^2(\Omega;\R^N)}^2  
= \inf_{\lambda\in \Lambda} 
\Ical(\lambda)\leq\Ical(\tilde \lambda) ,
\end{align*}
which finishes the proof.
\end{proof}

\begin{remark}\label{rem:stability1}
We give a simple counterexample to illustrate that minimizers for  
$\Ical$ may not exist in general. Take $\Lambda = (0,\infty)\subset \R$, a single data point \((u^c,u^\eta)\) with $u^{c}=u^{\eta}\not = 0$, and $\Rcal_{\lambda}(u)=\lambda \norm{u}^2_{L^2(\Omega)}$ for $\lambda\in \Lambda$. Then,  $\Jcal_\lambda(u)=\norm{u-u_\eta}^2_{L^2(\Omega)} + \lambda\norm{u}_{L^2(\Omega)}^2$ for $u\in L^2(\Omega)$ and $K_{\lambda}=\{u^{\eta}/(1+\lambda)\}=\{u^c/(1+\lambda)\}$, so that
\[
\Ical(\lambda) = \left(\frac{\lambda}{1+\lambda}\right)^2\norm{u^c}^2_{L^2(\Omega)},
\]
which does not have a minimizer on $\Lambda=(0,\infty)$. By the previous theorem,  the family must fail to be weakly $L^2$-stable. Indeed, $\Gamma(w\text{-}L^2)\text{-}\lim_{\lambda\to 0}\Jcal_\lambda$ coincides with the pointwise limit and is equal to $\norm{\cdot \ -u_\eta}_{L^2(\Omega)}^2$, which is not an element of $\{\Jcal_\lambda\}_{\lambda\in (0,\infty)}$. 
\end{remark}

 Theorem~\ref{theo:suff} is useful in many situations, including the basic case when the parameter set $\Lambda$ is a compact real interval. However, weak $L^2$-stability is not always guaranteed, as Remark~\ref{rem:stability1} illustrates. If, for instance, we have a sequence $(\lambda_k)_k$ converging to a point in $X$ outside $\Lambda$, then there is no reason to expect that 
\[
\Gamma(w\text{-}L^2)\text{-}\lim_{k\to \infty} \Jcal_{\lambda_k, j} = \Jcal_{\lambda, j}
\]
holds for some $\lambda \in \Lambda$. 

To overcome this issue and provide a more general existence framework, we will look at a suitable replacement of the bi-level scheme. In the following, we denote by $\overline{\Lambda}$ the closure of $\Lambda$ and suppose that for each $j\in\{1, \ldots, N\}$ and $\lambda\in \overline{\Lambda}$, the $\Gamma$-limits \begin{align}\label{eq:gammaextension}
\overline{\Jcal}_{\lambda, j} := \Gamma(w\text{-}L^2)\text{-}\lim_{\lambda'\to \lambda} \Jcal_{\lambda', j}
\end{align} 
exist, where \(\lambda'\) takes values on an arbitrary sequence in \(\Lambda\). We further set $$\overline{\Jcal}_{\lambda} := (\overline{\Jcal}_{\lambda,1}, \ldots, \overline{\Jcal}_{\lambda, N}):\overline{\Lambda}\to [0, \infty]^N.$$

Based on these definitions, we introduce $\overline{\Ical}:\overline{\Lambda}\to [0, \infty]$ as   the extension of the upper level functional $\Ical$ given by
\begin{align}\label{eq:extI}
\overline{\Ical}(\lambda) :=\min_{w\in \overline{K}_{\lambda}}\|w-u^{c}\|^2_{L^2(\Omega;\R^N)},
\end{align}
where $\overline{K}_{\lambda,j}:= {\rm argmin}_{u\in L^2(\Omega)} \overline{\Jcal}_{\lambda, j}(u)$ and $\overline{K}_{\lambda} := \overline{K}_{\lambda,1}\times \overline{K}_{\lambda,2} \times \dots \times \overline{K}_{\lambda,N}$ for $\lambda\in \overline{\Lambda}$. Observe that $\overline{K}_{\lambda,j}$ is $L^2$-weakly closed because the functional $\overline{\Jcal}_{\lambda,j}$, as a $\Gamma(w\text{-}L^2)$-limit by~\eqref{eq:gammaextension}, is $L^2$-weakly lower semicontinuous. Hence, the minimum in the definition of 
$\overline \Ical$ 
is actually attained. Notice that taking  constant
sequences in the parameter space in~\eqref{eq:gammaextension} and using the weak lower semicontinuity of the regularizers $\Rcal_{\lambda}$ in \eqref{eq:assumptions}, we conclude that  $\overline{\Jcal}_{\lambda}$ coincides with $\Jcal_{\lambda}$ whenever $\lambda\in \Lambda$.  In that sense, we can think of $\{\overline{\Jcal}_{\lambda}\}_{\lambda \in \overline{\Lambda}}$ 
as the extension of the family $\{\Jcal_{\lambda}\}_{\lambda \in \Lambda}$ to the closure of $\Lambda$.

All together, this leads to the extended bi-level problem
\begin{align}\label{trainingextended}\tag{$\overline{\Tcal}$}
\begin{split}
({\rm Upper\text{-}level}) \qquad& \text{Minimize} \ \ {\overline{\Ical}}(\lambda):=\min_{w \in \overline{K}_{\lambda}}\norm{w-u^c}^2_{L^2(\Omega;\R^N)} \ \ \text{over $\lambda \in \overline{\Lambda}$,} \\
({\rm Lower\text{-}level}) \qquad& \overline{K}_{\lambda} : = \Bigl\{w\in L^2(\Omega;\R^N): w_j\in \argmin_{u\in L^2(\Omega)} \overline{\Jcal}_{\lambda,
j}(u) \hbox{ for all } j\in\{1, \ldots, N\}\Bigr\}. 
\end{split}
\end{align}

The theorem below compares the extended upper level functional $\overline{\Ical}$ with the relaxation of $\Ical$ (after trivial extension to $\overline{\Lambda}$ by $\infty$), that is, with its lower semicontinuous envelope  $\Ical^{\rm rlx}: \overline{\Lambda} \to [0, \infty]$ given by 
\begin{align}\label{Irlx}
\Ical^{\rm rlx}(\lambda):=\inf\big\{\liminf_{k\to \infty} 
\Ical (\lambda_k): (\lambda_k)_k\subset \Lambda, \lambda_k\to \lambda \text{ in $\overline{\Lambda}$}\big\}.
\end{align}
As we will see, the key assumption to obtain the equality between $\overline{\Ical}$ and $\Ical^{\rm rlx}$ is the Mosco-convergence of the family of regularizers in \eqref{eq:Moscoconvergence}, which is stronger than the $\Gamma$-convergence of the reconstruction functionals in \eqref{eq:gammaextension}.
It even implies the Mosco-convergence
\begin{align*}
\overline{\Jcal}_{\lambda, j} =\text{Mosc}(L^2)\text{-}\lim_{\lambda'\to \lambda} \Jcal_{\lambda', j}
\end{align*}
and, in this case, the limit passage can be performed additively in the fidelity and regularizing term; thus,  for all $j\in \{1, \ldots, N\}$, we have
\begin{align}\label{eq67}
\overline{\Jcal}_{\lambda, j}(u)= \norm{u - u_j^\eta}_{L^2(\Omega)} + \overline{\Rcal}_\lambda(u) \qquad \text{for $u\in L^2(\Omega)$.}
\end{align}

\begin{theorem}\label{theo:relax2}
Consider the bi-level optimization problems~\eqref{training} and~\eqref{trainingextended}, assume \eqref{eq:gammaextension}, and recall the definitions in \eqref{eq:extI} and \eqref{Irlx}. 
Suppose in addition that
\begin{itemize}
\item[$(i)$] the Mosco-limits 
\begin{align}\label{eq:Moscoconvergence}
\overline{\Rcal}_{\lambda} := \text{\rm Mosc}(L^2)\text{-}\lim_{\lambda'\to \lambda} \Rcal_{\lambda'}
\end{align} 
exist for each $\lambda\in \overline{\Lambda}$, with $\lambda'$ taking values on sequences in $\Lambda$, and
\item[$(ii)$] $\overline{K}_\lambda$ is a singleton for every $\lambda\in \overline{\Lambda}\setminus \Lambda$.
\end{itemize}
Then, the extension $\overline{\Ical}$ of $\Ical$ to the closure $\overline{\Lambda}$ coincides with the relaxation of $\Ical$, i.e., $\overline{\Ical}=\Ical^{\rm rlx}$ on $\overline{\Lambda}$.
\end{theorem}

\begin{proof}
To show that $\overline{\Ical} \leq \Ical^{\rm rlx}$, we take  $ \lambda\in \overline{\Lambda}$
and let $(\lambda_k)_k\subset \Lambda$ with  $\lambda_k\to \lambda$ in $\overline\Lambda$ be an admissible sequence  for \(\Ical^{\rm rlx}(\lambda)\)
in~\eqref{Irlx}.
We  may even assume that 
 $\infty>\liminf_{k\to \infty} \Ical (\lambda_k)=\lim_{k\to \infty} \Ical (\lambda_k)$.
Then,  recalling \eqref{eq:argminK}  and  fixing $\delta>0$, we can find $w_k\in K_{\lambda_k}$ such that 
\begin{align*}
\lim_{k\to \infty} \Ical(\lambda_k) \geq \liminf_{k\to \infty} \norm{w_k - u^{c}}_{L^2(\Omega;\R^N)}^2-\delta.
\end{align*}
In particular, $({w_k})_k$ is uniformly bounded in $L^2(\Omega;\R^N)$, which allows us to extract an \(L^2\)-weakly converging subsequence (not relabeled) with limit $\bar w\in L^2(\Omega;\R^N)$.  By the properties of $\Gamma$-convergence,
we infer from \eqref{eq:gammaextension} that  $\bar w_j \in {\rm argmin}_{u\in L^2(\Omega)}\overline{\Jcal}_{ \lambda, j}(u) $ for all $j\in\{1,\cdots,N\}$;  in other words,
\(\bar w \in\overline{K}_{\lambda} \). Thus,
\begin{align*}
\lim_{k\to \infty} 
\Ical (\lambda_k) \geq  \norm {\bar w - u^{c}} _{L^2(\Omega;\R^N)}^2-\delta \geq \overline \Ical( \lambda) -\delta.
\end{align*}
By letting \(\delta\to0\) first, and then taking the infimum over all admissible sequences for \(\Ical^{\rm rlx}(\lambda)\) in~\eqref{Irlx}, it follows that $\overline{\Ical}(\lambda) \leq \Ical^{\rm rlx}(\lambda)$.

To prove the reverse inequality, it suffices to consider $ \lambda\in \overline{\Lambda}\setminus \Lambda$ and find a sequence $(\lambda_k)_k\subset
\Lambda$ converging to $\lambda$ in $\overline{\Lambda}$ and satisfying $\liminf_{k\to \infty} 
\Ical(\lambda_k)  \leq \overline{\Ical}(\lambda)$. To that end, take any $(\lambda_k)_k\subset \Lambda$ with $\lambda_k\to \lambda$ in $\overline{\Lambda}$, and let
 $w_k \in K_{\lambda_k}$ for $k\in \N$. Recalling $(ii)$, denote by  $w_\lambda=(w_{\lambda,
1}, \ldots, w_{\lambda, N})$  the unique element in $\overline{K}_\lambda$. Then, using ~\eqref{eq:gammaextension} and the equi-coercivity of $(\Jcal_{\lambda})_{\lambda\in \Lambda}$, we obtain by the theory of $\Gamma$-convergence (see ~\cite[Corollary 7.24]{Dal93})  that $(w_k)_k$ converges weakly in $L^2(\Omega;\R^N)$ to $w_\lambda$; moreover, it holds for all $j\in \{1, \ldots, N\}$  that
 \begin{align}\label{eq:con4}
 \Jcal_{\lambda_k, j}(w_{k, j})\to \overline{\Jcal}_{\lambda, j}(w_{\lambda, j})\qquad \text{as $k\to \infty$.}
 \end{align} 
 
The following shows that $(w_k)_k$ converges even strongly in $L^2(\Omega;\R^N)$. Indeed, fixing $j\in\{1, \ldots, N\}$, we infer from~\eqref{eq:con4} along with the Mosco-convergence of the regularizers in $(i)$ and~\eqref{eq67} that 
\begin{align*}
\norm{w_{\lambda, j} - u_j^\eta}^2_{L^2(\Omega)} +\overline{\Rcal}_\lambda(w_{\lambda, j}) &= \overline{\Jcal}_{\lambda, j}(w_{\lambda, j}) = \lim_{k\to \infty} \Jcal_{\lambda_k, j} (w_{k, j}) \nonumber \\& = \lim_{k\to \infty} \Bigl[\norm{w_{k, j} - u_j^\eta}^2_{L^2(\Omega)} +{\Rcal}_{\lambda_k}(w_{k, j})\Bigr] \\ 
& \geq \limsup_{k\to \infty} \norm{w_{k, j} - u_j^\eta}^2_{L^2(\Omega)} + \overline{\Rcal}_\lambda(w_{\lambda, j}). \nonumber
\end{align*}
 Hence,  \(\norm{w_{\lambda, j} - u_j^\eta}^2_{L^2(\Omega)}\geq \limsup_{k\to \infty} \norm{w_{k, j} - u_j^\eta}^2_{L^2(\Omega)} \), which  together with the weak lower semicontinuity of the $L^2$-norm yields 
 \begin{align*}
 \lim_{k\to \infty} \norm{w_{k, j}-u_j^\eta}_{L^2(\Omega)}^2= \norm{w_{\lambda, j} -u_j^\eta}_{L^2(\Omega)}^2;
 \end{align*}
 thus, $w_{k}\to w_{\lambda}$ strongly in $L^2(\Omega;\R^N)$ using the combination of weak convergence and convergence of norms. 
With this, we finally conclude that 
\begin{align*}
\liminf_{k\to \infty} 
\Ical(\lambda_k)  & 
\leq \liminf_{k\to \infty} \norm{w_k- u^{c}}_{L^2(\Omega;\R^N)}^2   =  \norm{w_\lambda- u^{c}}_{L^2(\Omega;\R^N)}^2 =  \min_{w\in \overline{K}_\lambda}\norm{w- u^{c}}_{L^2(\Omega;\R^N)}^2 = \overline{\Ical}(\lambda),
\end{align*}
 finishing the proof. 
\end{proof}

\begin{remark}\label{rem:MoscovsGamma}
By inspecting the proof, it becomes clear that the estimate $\overline{\Ical} \leq \Ical^{\rm rlx}$  holds without the additional assumptions $(i)$ and $(ii)$ from the previous theorem;  in other words, $\overline{\Ical}$ always provides a lower bound for the relaxation of $\Ical$.
\end{remark}

The identity $\overline{\Ical}=\Ical^{\rm rlx}$ mail fail if either of the assumptions $(i)$ or $(ii)$ in Theorem~\ref{theo:relax2} is  dropped as the following example shows.  
\begin{example}\label{ex:optimal}
a) To see why $(i)$ is necessary, consider $\Lambda =(0,1]$, a single data point \((u^c,u^\eta)\) with $u^c=u^{\eta}=0$, and $$\Rcal_{\lambda}=\frac{1}{\lambda}\norm{\cdot\,-v_{\lambda}}^2_{L^2(\Omega)} \qquad \text{with 
 $v_{\lambda}=v(\cdot/\lambda)\in L^2(\Omega)$}$$
 for a given $v \in L^{\infty}(\R^n)$ with the properties that $v$ is $(0,1)^n$-periodic, $v \in \{-1,1\}$ almost everywhere,  and $\int_{(0,1)^n} v\,\dx=0$. Under these specifications, the $\Gamma$-limits $\overline{\Jcal}_\lambda=\Gamma(w\text{-}L^2)\text{-}\lim_{\lambda'\to \lambda}\Jcal_{\lambda'}$ (cf.~\eqref{eq:gammaextension} and~\eqref{general_model}) exist and are given by
\begin{align}\label{Lcal_remark}
\overline{\Jcal}_{\lambda}(u)=\begin{cases}
\norm{u}^2_{L^2(\Omega)}+\frac{1}{\lambda}\norm{u-v_{\lambda}}_{L^2(\Omega)}^2 & \text{for $\lambda \in (0,1],$}\\ 
\abs{\Omega}+\chi_{\{0\}}(u) & \text{for $\lambda=0$},
\end{cases}
\end{align}
where $\chi_{E}$ denotes the indicator function of a set $E \subset L^2(\Omega)$, i.e.,
\begin{equation*}
\chi_{E}(u)=\begin{cases}
0 &\text{if $u \in E$},\\
\infty &\text{if $u \not \in E$},
\end{cases} \qquad \text{for $u \in L^2(\Omega)$}.
\end{equation*} 

The non-trivial case is  when $\lambda=0$. In this case,  we observe that we can take   $(v_{\lambda'})_{\lambda'}$ as a recovery sequence for $u=0$ because it converges  weakly  in $L^2(\Omega)$ as $\lambda'\to 0$ to  $\int_{(0,1)^n}v\, \dx =0$ by the Riemann--Lebesgue lemma for periodically oscillating sequences. For the liminf inequality, let $u_{\lambda'} \rightharpoonup u$ as $\lambda' \to 0$ and suppose without loss of generality that $\sup_{\lambda'}\Rcal_{\lambda'}(u_{\lambda'})<\infty$. Then, $u_{\lambda'} = v_{\lambda'}+r_{\lambda'}$ with $r_{\lambda'} \to 0$ in $L^2(\Omega)$ as $\lambda' \to 0$, which implies $u=0$ and, recalling that $v\in \{-1,1\}$ almost everywhere,
\[
\liminf_{\lambda' \to 0}\Jcal_{\lambda'}(u_{\lambda'})\geq \lim_{\lambda' \to 0}\norm{v_{\lambda'}+r_{\lambda'}}^2_{L^2(\Omega)}=\lim_{\lambda' \to 0}\norm{v_{\lambda'}}^2_{L^2(\Omega)}=\abs{\Omega} = \overline{\Jcal}_0(0),
\] 
 which completes the proof of~\eqref{Lcal_remark} when \(\lambda=0\). 

In view of~\eqref{Lcal_remark}, one can now read off that $K_{\lambda}=\overline{K}_\lambda=\{v_{\lambda}/(1+\lambda)\}$ for $\lambda \in (0,1]$ and $\overline{K}_0=\{0\}$. In particular,  condition $(ii)$ on the uniqueness of minimizers of the extended lower-level problem is fulfilled here. Hence, 
\begin{align}\label{Ical_ex27}
\Ical(\lambda) = \Bigl(\frac{1}{1+\lambda}\Bigr)^2\abs{\Omega} 
\end{align}
for $\lambda\in (0,1],$
and
\[ \overline{\Ical}(\lambda) = \begin{cases}
\displaystyle \Bigl(\frac{1}{1+\lambda}\Bigr)^2\abs{\Omega} &\text{if $\lambda \in (0,1]$},\\
0 &\text{if $\lambda =0$} \quad 
\end{cases}
\]
for $\lambda\in [0,1]$. It is immediate to see from~\eqref{Ical_ex27} that 
$$\overline{\Ical}(0) =0 <  |\Omega|=\Ical^{\rm rlx}(0).$$ Notice  that this example hinges on the fact that the minimizers $v_{\lambda}/(1+\lambda)$ only converge weakly as $\lambda \to 0$, which, in view of the proof of Theorem~\ref{theo:relax2}, implies that the family of regularizers $\{\Rcal_\lambda\}_{\lambda\in \Lambda}$ does not Mosco-converge in $L^2(\Omega)$ in the sense of~\eqref{eq:Moscoconvergence},   thus failing to satisfy $(i)$. 
\smallskip

b) For the necessity of $(ii)$, consider $\Lambda=(0,1]$, a single data point \((u^c,u^\eta)\) with $u^c=0$ and $\norm{u^{\eta}}^2_{L^2(\Omega)}=1$, and
\[
\Rcal_{\lambda}(u) =\begin{cases}
\lambda &\text{if $u=0$},\\
1 &\text{if $u \not = 0$}.
\end{cases}
\]
While it is straightforward to check that condition \((i)\) in Theorem~\ref{theo:relax2} regarding the Mosco-limits of $\{\Rcal_\lambda\}_{\lambda \in \Lambda}$ is satisfied with
\[
\overline{\Rcal}_{\lambda}(u) =\begin{cases}
\lambda &\text{if $u=0$},\\
1 &\text{if $u \not = 0$}
\end{cases}
\]
for $\lambda \in [0,1]$, which clearly coincides with $\Rcal_{\lambda}$ for $\lambda \in \Lambda=(0,1]$, condition $(ii)$ fails. Indeed, it follows from~\eqref{eq67} that $\overline{\Jcal}_{\lambda}(u^{\eta})= \overline{\Rcal}_\lambda(u^\eta)=1$ and $\overline{\Jcal}_{\lambda}(0)=\norm{u^{\eta}}^2_{L^2(\Omega)}+\lambda = 1+\lambda$ for all $\lambda \in [0,1]$.  Consequently, for $\lambda\in (0,1]$, we have $\overline{\Jcal}_\lambda=\Jcal_\lambda$  and   $u^\eta$ is its  unique minimizer; in contrast, for \(\lambda=0\), $\overline{\Jcal}_0$ has two minimizers, namely  
 $\overline{K}_{0}=\{u^{\eta},0\}=\{u^{\eta},u^c\}$.  
Finally, we observe that the conclusion of Theorem~\ref{theo:relax2} fails here because\[
\overline{\Ical}(0)=0 \qquad \text{and} \qquad \Ical(\lambda)= \norm{u^c-u^{\eta}}^2_{L^2(\Omega)} =1 \ \ \text{for all $\lambda \in (0,1]$},
\]
which yields $\overline{\Ical}(0)=0 < 1= \Ical^{\rm rlx}(0)$.
\end{example}

The following result is a direct consequence of Theorem~\ref{theo:relax2} and standard properties of relaxation.

\begin{corollary}\label{cor:relax}
Under the assumptions of 
 Theorem~\ref{theo:relax2} and if $\overline{\Lambda}$ is compact, it holds that:
\begin{itemize}
\item[$(i)$] The extension $\overline{\Ical}$ has at least one minimizer and
\[
\min_{\overline{\Lambda}} \overline{\Ical} = \inf_{\Lambda} \Ical.
\]
\item[$(ii)$] Any minimizing sequence $(\lambda_k)_k \subset \Lambda$ of $\Ical$ converges up to subsequence to a minimizer $\lambda \in \overline{\Lambda}$ of $\overline{\Ical}$.

\item[$(iii)$] If $\lambda \in \Lambda$ minimizes $\overline{\Ical}$, then $\lambda$ is also a minimizer of $\Ical$.
\end{itemize}

\end{corollary}

 We conclude this section on the theoretical framework with a brief comparison with related works on optimal control problems. 
 By setting $K=\{(w,\lambda) \in L^2(\Omega)\times \Lambda \,:\, w \in K_{\lambda}\}$, the bi-level optimization problem~\eqref{training} can be equivalently rephrased into minimizing
\[
\widehat{\Ical}(u,\lambda) = \norm{u-u^c}^2_{L^2(\Omega)} + \chi_{K}(u,\lambda),\qquad (u, \lambda)\in L^2(\Omega) \times \Lambda, 
\]
as a functional of two variables; observe that 
\begin{align*}
\Ical(\lambda)= \inf_{w\in L^2(\Omega)} \widehat{\Ical}(w, \lambda). 
\end{align*}
Similar functionals and their relaxations have been studied in the literature, including~\cite{BuD82, But87, BBF93}. Especially the paper~\cite{BBF93} by Belloni, Buttazzo, \& Freddi, where the authors propose to extend the control space to its closure and find a description of the relaxed optimal control problem, shares many parallels with our results.
Apart from some differences in the assumptions and abstract set-up, the main reason why their results are not applicable here is the continuity condition of the cost functional with respect to the state variable \cite[Eq.~(2.11)]{BBF93}. In our setting, this would translate into weak continuity of the $L^2$-norm, which is clearly false. The argument in the proof of Theorem~\ref{theo:relax2} exploiting the Mosco-convergence of the regularizers (see \eqref{eq:Moscoconvergence}) is precisely what circumvents this issue.

\section{Learning the optimal weight of the regularization term}\label{sec:weight}
In this section, we study the optimization of a weight factor, often called tuning parameter, in front of
a fixed regularization term. 
Such tuning parameters are typically employed in practical implementations of variational denoising models to adjust the best level of regularization. This setting constitutes a simple,
yet non-trivial, application of our general theory and therefore helps to exemplify the abstract results from the previous section. 

As above, $\Omega\subset \R^n$ is a bounded open set and $u^{c}$, $u^{\eta}\in L^2(\Omega;\R^N)$ are the given data representing pairs of clean
and noisy images. We take \(\Lambda=(0,\infty)\) describing  the range of a weight factor and, to distinguish
the various parameters throughout this paper, denote by \(\alpha\) an
arbitrary point in \(\overline\Lambda=[0,\infty]\). 
For a fixed map $\Rcal:L^2(\Omega)\to
[0, \infty]$ with the properties that
\begin{enumerate}
 \item[(H$1_{\alpha}$)] $\Rcal$ is convex,
vanishes exactly on constant functions, and ${\rm Dom}\, \Rcal$ is dense
in $L^2(\Omega)$, 
  \item[(H$2_{\alpha}$)] $\Rcal$ is lower semicontinuous on $L^2(\Omega)$,
 \end{enumerate}
 we define the weighted regularizers   
\begin{align}\label{Ralpha}
\Rcal_\alpha=\alpha\Rcal \qquad \text{ for $\alpha\in (0, \infty)$}.
\end{align}
 Note that~(H$1_{\alpha}$) and (H$2_{\alpha}$) imply that the family $\{\Rcal_{\alpha}\}_{\alpha \in (0, \infty)}$ satisfies \eqref{eq:assumptions} because convexity and lower semicontinuity yield weak lower semicontinuity, making this setting match with the framework of Section~\ref{sec:general}.

Following the definition of the training scheme~\eqref{training}, we introduce here for $\alpha\in (0, \infty)$ and $j\in \{1, \ldots, N\}$ the reconstruction functionals
  \begin{align*}
  \Jcal_{\alpha, j}(u) = \norm{u-u_j^{\eta}}_{L^2(\Omega)}^2 + \Rcal_\alpha(u) \qquad \text{for $u\in L^2(\Omega)$,}
  \end{align*} 
  cf.~\eqref{general_model},
 and consider accordingly the upper level functional  $\Ical:(0, \infty)\to [0, \infty)$ given by
  \begin{align}\label{training_sec2}
  \Ical(\alpha) = \inf_{w\in K_\alpha}\|w-u^{c}\|^2_{L^2(\Omega;\R^N)}
\enspace \text{ for  \(\alpha\in(0,\infty)\)}, \quad
  \end{align} 
  with $K_\alpha=K_{\alpha, 1}\times \ldots\times K_{\alpha, N}$ and $K_{\alpha, j} = \argmin_{u\in L^2(\Omega)} \Jcal_{\alpha, j}(u)$, cf.~ \eqref{eq:argminK}.
Further, the following set of hypotheses on 
the training data will play a crucial role for our main result in this section (Theorem~\ref{thm:pos}):  
 \begin{enumerate}
\item[(H$3_{\alpha}$)] It holds that
  \begin{align*}
\sum_{j=1}^N \Rcal(u^{c}_j)<\sum_{j=1}^N\Rcal(u^{\eta}_j);
  \end{align*}
  \item[(H$4_{\alpha}$)] the data $u^\eta$ and $u^c$ satisfy  \begin{align*}
 \norm{u^{\eta}-u^{c}}_{L^2(\Omega;\R^N)}^2< \normB{\dashint_{\Omega} u^{\eta}\,\dx
- u^{c}}_{L^2(\Omega;\R^N)}^2.
  \end{align*} 
 \end{enumerate}

 \begin{remark}[Discussion of the hypotheses (H\boldmath{$1_{\alpha}$})--(H\boldmath{$4_{\alpha}$})]\label{rem:H0H3}
 a)  Note that  (H$1_{\alpha}$) implies that the set of minimizers for the reconstruction functionals, $K_{\alpha}$, has cardinality one, owing to the convexity of $\Rcal$ and the strict convexity of the fidelity term, considering also that $\Jcal_{\alpha, j}\not\equiv \infty$. In the following, we write $w^{(\alpha)}=(w^{(\alpha)}_1, \ldots, w^{(\alpha)}_N)\in L^2(\Omega;\R^N)$ for the single element of $K_{\alpha}$, i.e.,
$K_\alpha=\{w^{(\alpha)}\}$.\smallskip 

b) An example of a nonlocal regularizer satisfying (H$1_{\alpha}$) and (H$2_{\alpha}$) is
\begin{align*}
\mathcal{R}(u):=\int_{\Omega}\int_{\Omega}a(x,y)\,g(u(x)-u(y))\,\dx\,\dy
\quad \text{for $u\in L^2(\Omega)$,}
\end{align*} 
where $g:\R \to [0,\infty)$ is a convex function such that $g^{-1}(0)=\{0\}$ and $a:\Omega\times 
\Omega\to[0,\infty]$ 
is a suitable kernel ensuring that $C_c^{\infty}(\Omega)\subset \textrm{Dom}\,\Rcal$. As an explicit choice, one can take $g(t)=t^p$ for $t\in \R$ and $a(x, y)=|y-x|^{-n-sp}$ for $x,\, y\in \Omega$ with some $s\in (0,1)$ and $p\geq 1$, which corresponds to a fractional Sobolev regularization. \smallskip

c) Assumption (H$3_{\alpha}$) asserts that the regularizer penalizes the noisy images more
than the clean ones on average. This is a  natural condition because any good regularizer
should reflect the prior knowledge on the training data, favoring
the clean images. 
\smallskip

d) The second condition on the data,  (H$4_{\alpha}$), means that the noisy image lies closer to the
clean image than its mean value, which can be considered a reasonable assumption in the case
of moderate noise and a non-trivial ground truth.  Indeed, suppose the noise is bounded by $\norm{u_j^{\eta} - u_j^{c}}_{L^2(\Omega)}\leq
\delta$ for all $j\in\{1,\dots,N\}$  and some $\delta>0$; then, (H$4_{\alpha}$) is satisfied if
\begin{align*}
 \normB{\dashint_{\Omega}u^{c}_j\,\dx-u^{c}_j} _{L^2(\Omega)}> \delta\big(1+|\Omega|^{-\frac12}\big)\enspace
\text{ for all } j\in\{1,...,N\}
 \end{align*}
 because  
\begin{align*}
 \normB{\dashint_{\Omega}u^{\eta}_j\,\dx-u^{c}_j} _{L^2(\Omega)}&\geq \normB{\dashint_{\Omega}u^{c}_j\,\dx-u^{c}_j}
_{L^2(\Omega)}-\normB{\dashint_{\Omega}(u_j^{\eta}-u_j^{c})\,\dx}_{L^2(\Omega)} \\ &> 
\delta\big(1+|\Omega|^{-\frac12}\big)- |\Omega|^{-\frac12} \norm{u_j^{\eta} - u_j^{c}}_{L^2(\Omega)}\\ &\geq
\delta\geq\norm{u_j^{\eta}
- u_j^{c}}_{L^2(\Omega)},
 \end{align*}
 where the second inequality is due to Jensen's inequality.
 \end{remark}

     \color{black}
Next, we prove that the assumptions (H$1_{\alpha}$)--(H$4_{\alpha}$) on the regularization term and
on the training set give rise to optimal weight parameters that stay away from the
extremal regimes, $\alpha=0$ and $\alpha=\infty$.
Thus, in this case, the bi-level parameter optimization procedure preserves
the structure of the original denoising model. 
 \begin{theorem}[Structure preservation]
 \label{thm:pos}
Suppose that {\rm (H$1_{\alpha}$)--(H$4_{\alpha}$)} hold. 
 Then, the learning scheme corresponding to the minimization of $\Ical$ in~\eqref{training_sec2}
admits a
solution $\bar{\alpha}\in (0,\infty)$.
 \end{theorem}
 
A related statement in the same spirit can be found in~\cite[Theorem~1]{DeScVa16}, although some of the details of the proof were not entirely clear to us.
Our proof of Theorem~\ref{thm:pos} is based on a different approach and hinges on the following two lemmas,
the first of which determines the Mosco-limits of the regularizers, and thereby provides an explicit formula of the extension $\overline{\Ical}$
of $\Ical$ as introduced in \eqref{eq:extI}.
 
\begin{proposition}[Mosco-convergence of the regularizer]\label{lem:relax1}
Let $\Rcal:L^2(\Omega)\to[0, \infty]$ satisfy {\rm (H$1_{\alpha}$)} and {\rm (H$2_{\alpha}$)}, and let $\{\Rcal_{\alpha}\}_{\alpha\in (0, \infty)}$ be as in~\eqref{Ralpha}. Then,
\begin{equation}\label{eq:moscoweight}
\overline{\Rcal}_{\alpha} := \text{\rm Mosc}(L^2)\text{-}\lim_{\alpha'\to \alpha} \Rcal_{\alpha'}=\begin{cases}
\Rcal_{\alpha} & \text{for $\alpha \in (0, \infty)$},\\
0 & \text{for $\alpha=0$},\\
\chi_{C} & \text{for $\alpha=\infty$},
\end{cases} 
\end{equation}
for $\alpha \in [0,\infty]$, where $\chi_C$ is the indicator function of $C:=\{u \in L^2(\Omega): u\ \mathrm{ is \ constant}\}$.
\end{proposition}
\begin{proof} 
Using standard arguments, we show that  the 
Mosco-limit of $(\Rcal_{\alpha_k})_k$ exists for every sequence $(\alpha_k)_k$ of positive real
numbers with $\alpha_k\to \alpha \in [0,\infty]$, and corresponds to the right hand side of \eqref{eq:moscoweight}. \smallskip 

\textit{Case 1: $\alpha \in (0,\infty)$.}  Using (H$2_{\alpha}$) for the liminf inequality and a constant recovery sequence for the upper bound, we conclude that the Mosco-limit of $(\Rcal_{\alpha_k})_k$ coincides with $\Rcal_{\alpha}$.  
 \smallskip

\textit{Case 2: $\alpha=0$.} The liminf inequality is trivial. For the
recovery sequence, take $u \in L^2(\Omega)$ and let $(u_k)_k \subset {\rm
Dom}\,\Rcal$ converge strongly to $u$ in \(L^2(\Omega) \), which is feasible
due to (H$1_{\alpha}$). By possibly repeating certain entries of the sequence $(u_k)_k$ (not relabeled),
one can slowdown the speed at which $\Rcal(u_k)$ potentially blows up and assume that $\alpha_{k}\Rcal(u_k) \to 0$ as $k \to \infty$.
Thus,
\[
\lim_{k \to \infty} \Rcal_{\alpha_k}(u_k)= \lim_{k \to \infty} \alpha_k\Rcal(u_k) =0. 
\]

\textit{Case 3: $\alpha = \infty$.} The limsup inequality follows by choosing
constant recovery sequences. For the proof of the lower bound, consider 
 $u_k\weakly  u$ in $L^2(\Omega)$ with $r:=\sup_{k\in \N}\alpha_k\Rcal(u_k) = \sup_{k\in \N} \Rcal_{\alpha_k}(u_k)<\infty$. Then, along with the weak lower semicontinuity of $\Rcal$ (see Remark~\ref{rem:H0H3}\,a)),
\begin{align*}
\Rcal(u) \leq \liminf
_{k\to \infty} \Rcal(u_k) &\leq\lim_{k\to
\infty} \frac{r}{\alpha_k} = 0.
\end{align*}
This shows that $\Rcal(u)=0$, which implies by the assumption on the zero level set
of $\Rcal$ in (H$1_{\alpha}$) that $u$ is constant, i.e., $u\in C$.  
\end{proof}

As a consequence of the previous lemma, we deduce that the extension $\overline{\Ical}:\overline{\Lambda} \to [0,\infty]$ of $\Ical$ in the sense of~\eqref{eq:extI} can be explicitly determined as
\begin{equation}\label{Ical_relax_alpha} 
\overline{\Ical}(\alpha) = \begin{cases}\Ical(\alpha)
& \text{for $\alpha\in (0, \infty)$,}\\
 \norm{u^{\eta}-u^{c}}^2_{L^2(\Omega;\R^N)} & \text{for $\alpha=0$,}\\[0.1cm]
 \normB{\displaystyle \dashint_\Omega u^{\eta} \,\dx
- u^{c}}^2_{L^2(\Omega;\R^N)} & \text{for $\alpha=\infty$. }
\end{cases}
\end{equation}
Indeed, a straight-forward calculation of the unique componentwise minimizer of the extended reconstruction functionals $\overline{\Jcal}_{\alpha}$ at the boundary points $\alpha=0$ and $\alpha=\infty$ leads to
\begin{align*}
\overline{K}_0=\{u^\eta\}\quad \text{ and }\quad \overline{K}_\infty=\Bigl\{\dashint_\Omega u^{\eta}\,\dx\Bigr\}. 
\end{align*}
Since the assumptions $(i)$ and $(ii)$ of Theorem~\ref{theo:relax2} are satisfied, 
$\overline{\Ical}$ 
coincides with the relaxation $\Ical^{\rm rlx}$. 
By Corollary~\ref{cor:relax}\,$(i)$, $\overline{\Ical}$ attains its minimum at some $\bar{\alpha} \in [0,\infty]$. The degenerate cases $\bar{\alpha
} \in \{0,\infty\}$ cannot be excluded a priori, but the next lemma shows that the minimum is attained in the interior $(0,\infty)$ under suitable assumptions on the training data. \color{black}
\begin{lemma}
\label{lemma:not-min} Suppose that~{\rm (H$1_{\alpha}$)} and {\rm (H$2_{\alpha}$)} hold, and let $K_\alpha=\{w^{(\alpha)}\}$ with $w^{(\alpha)}=(w_1^{(\alpha)}, \ldots, w_N^{\alpha})\in L^2(\Omega;\R^N)$ for $\alpha\in (0, \infty)$, cf.~Remark~\ref{rem:H0H3}\,a).
\begin{itemize}
\item[$(i)$] Under the additional assumption {\rm (H$3_{\alpha}$)},
there exists ${\alpha}\in (0,\infty)$ such that 
\begin{equation*}
\|w^{(\alpha)}-u^{c}\|_{L^2(\Omega;\R^N)}^2<\|u^{\eta}-u^{c}\|_{L^2(\Omega;\R^N)}^2.
\end{equation*}

\item[$(ii)$] Under the additional assumption {\rm (H$4_{\alpha}$)},
there exists  $\alpha_0\in (0,\infty)$  such that, for all $\alpha\in (0,\alpha_0)$, 
\begin{equation}
\label{eq:smaller_average}
\|w^{(\alpha)}-u^{c}\|_{L^2(\Omega;\R^N)}^2< \Bigl\| \dashint_{\Omega} u^{\eta}\,\dx
-u^{c}\Bigr\|_{L^2(\Omega;\R^N)}^2.
\end{equation}
\end{itemize}
\end{lemma}

\begin{proof} 
We start by providing two useful auxiliary results about the asymptotic behavior of the reconstruction vector $w^{(\alpha)}$ as $\alpha$ tends to zero; precisely,
 \begin{align}\label{walphato0}
\lim_{\alpha\to 0} \norm{w^{(\alpha)}-u^{\eta}}_{L^2(\Omega;\R^N)}
= 0\qquad  \text{and} \qquad\lim_{\alpha\to 0}\Rcal(w_j^{(\alpha)})=
\Rcal(u^{\eta}_j) \ \text{for every $j\in \{1, \ldots, N\}$}. 
\end{align}
Fix $j\in \{1, \ldots, N\}$ and let $(\alpha_k)_k\subset (0,\infty)$ be such that $\alpha_k\to 0$ as
$k\to \infty$. Take $u \in \mathrm{Dom}\,\Rcal$ with $\norm{u-u^\eta_j}^2_{L^2(\Omega)} \leq \varepsilon$ for some $\varepsilon>0$, which is possible by (H$1_{\alpha}$). Then, the minimality of $w_j^{(\alpha_k)}$
for $\Jcal_{\alpha_k, j}$ yields
\begin{align*}
\norm{w_j^{(\alpha_k)}-u_j^{\eta}}^2_{L^2(\Omega)}  \leq \Jcal_{\alpha_k,
j}(w_j^{(\alpha_k)})\leq \Jcal_{\alpha_k, j}(u) = 
 \norm{u-u^{\eta}_j}_{L^2(\Omega)}^2+\alpha_k \Rcal(u)\leq \varepsilon +\alpha_k \Rcal(u). 
\end{align*}
Since $\Rcal(u) < \infty$, we find
\[
\limsup_{k \to \infty} \,\norm{w_j^{(\alpha_k)}-u_j^{\eta}}^2_{L^2(\Omega)} \leq \varepsilon,
\]
which proves the first part of~\eqref{walphato0} due to the arbitrariness of $\varepsilon$.  
Exploiting the minimality of $w_j^{(\alpha)}$ for $\Jcal_{\alpha,j}$ again with $\alpha\in (0, \infty)$ entails
 \begin{align*}
\alpha \Rcal(w_j^{(\alpha)}) = \Rcal_\alpha(w_j^{(\alpha)})  \leq  \Jcal_{\alpha,
j}(w_j^{(\alpha)})\leq \Jcal_{\alpha, j}(u_j^{\eta}) = 
 \Rcal_\alpha(u_j^{\eta}) = \alpha \Rcal(u_j^{\eta});
\end{align*} 
 hence, $\Rcal(w_j^{(\alpha)})\leq \Rcal(u_j^{\eta})$ and,
together with the first part of~\eqref{walphato0} and the lower semicontinuity of $\Rcal$ by (H$2_{\alpha}$), it follows then that
\begin{align*}
\Rcal(u_j^{\eta})\geq  \limsup_{k\to \infty}\Rcal(w_j^{(\alpha_k)})\geq \liminf_{k\to
\infty}\Rcal(w_j^{(\alpha_k)})\geq \Rcal(u_j^{\eta}).
\end{align*}
Thus, $\lim_{k\to \infty}\Rcal(w_j^{(\alpha_k)})=\Rcal(u_j^{\eta})$, showing the second part of~\eqref{walphato0}. \smallskip

Regarding~\((i)\), we observe that
the minimality of $w_j^{(\alpha)}$
for $\Jcal_{\alpha, j}$ for any $\alpha\in (0, \infty)$ and $j\in \{1, \ldots, N\}$ imposes the necessary condition 
$0\in \partial \Jcal_{\alpha, j}(w_j^{(\alpha)})$
or, equivalently, 
\begin{align*}
2(u_j^{\eta}-w_j^{(\alpha)})\in \partial \Rcal_\alpha(w_j^{(\alpha)}) = \alpha
\partial \Rcal(w_j^{(\alpha)}),
\end{align*}
where $\partial \Ccal(u)\in L^2(\Omega)'\cong L^2(\Omega)$ is the subdifferential
of a convex function $\Ccal:L^2(\Omega)\to [0,\infty]$ at $u\in L^2(\Omega)$. 
Then, 
\begin{align*}
\|u_j^{\eta}-u_j^{c}\|_{L^2(\Omega)}^2-\|w_j^{(\alpha)}-u_j^{c}\|_{L^2(\Omega)}^2&=
2\langle u^{\eta}_j-w^{(\alpha)}_j,w^{(\alpha)}_j-u^{c}_j\rangle_{L^2(\Omega)} +\|w^{(\alpha)}_j-u^{\eta}_j\|^2_{L^2(\Omega)}\\ 
& \geq \Rcal_\alpha(w_j^{(\alpha)}) -\Rcal_\alpha(u_j^{c}) = \alpha\bigl(\Rcal(w_j^{(\alpha)})
-\Rcal(u_j^{c})\bigr),
\end{align*}
where $\langle\cdot, \cdot\rangle_{L^2(\Omega)}$ denotes the standard $L^2(\Omega)$-inner product.
 \color{black} 
 Summing both sides over $j \in \{1,\dots,N\}$ results in
\[
\norm{u^{\eta}-u^c}^2_{L^2(\Omega;\R^N)} - \norm{w^{(\alpha)}-u^c}^2_{L^2(\Omega;\R^N)} \geq \alpha \sum_{j=1}^N \bigl(\Rcal(w^{(\alpha)}_j)-\Rcal(u^c_j)\bigr).
\] 
By (H$3_{\alpha}$) in combination with the second part of~\eqref{walphato0}, there exists $\alpha_0>0$ such that 
\[
\textstyle \sum_{j=1}^N\Rcal(w_j^{(\alpha)})
> \sum_{j=1}^N \Rcal(u_j^{c})
\]
for all $\alpha\in(0,\alpha_0)$, so that choosing $\bar{\alpha}\in (0, \alpha_0)$ concludes the proof of $(i)$. \smallskip

To show $(ii)$, we exploit the first limit in~\eqref{walphato0}. Due to~(H$4_{\alpha}$), it follows then for any $(\alpha_k)_k$ of positive real
numbers with $\alpha_k\to 0$ as $k\to \infty$ that
\begin{align*}
\limsup_{k\to \infty}\,\|w^{(\alpha_k)}-u^{c}\|_{L^2(\Omega;\R^N)} &\leq \limsup_{k\to
\infty}\,\|w^{(\alpha_k)}-u^{\eta}\|_{L^2(\Omega;\R^N)}+ \|u^{\eta}-u^{c}\|_{L^2(\Omega;\R^N)}\\
&<\normB{ \dashint_{\Omega} u^{\eta}\,\dx-u^{c}}_{L^2(\Omega;\R^N)},
\end{align*}
which gives rise to~\eqref{eq:smaller_average} for all \(k\) sufficiently large.
\end{proof}

\color{black}
\begin{proof}[Proof of Theorem~\ref{thm:pos}] 
Since $\overline{\Ical}$ in~\eqref{Ical_relax_alpha} attains its infimum at a point $\bar{\alpha}\in (0, \infty)$ 
by Lemma~\ref{lemma:not-min}, we conclude from Corollary~\ref{cor:relax}\,$(iii)$ that $\bar{\alpha}$ is also a minimizer of $\Ical$. 
\end{proof}

Let us finally remark that the assumptions~(H$3_{\alpha}$) and (H$4_{\alpha}$) on the training data are necessary to obtain structure preservation in the sense of Theorem~\ref{thm:pos}.

\begin{remark} 
To see that~(H$3_{\alpha}$) and (H$4_{\alpha}$) can generally not be dropped, consider, for example, a regularizer $\Rcal:L^2(\Omega)\to[0,\infty]$ that satisfies (H$1_{\alpha}$) and (H$2_{\alpha}$) and is $2$-homogeneous, i.e., $\Rcal(\mu u)=\mu^2\Rcal(u)$ for all $u \in L^2(\Omega)$ and $\mu \in \R$. With a single, non-constant noisy image $u^{\eta} \in L^2(\Omega)$, so that $\Rcal(u^{\eta})\not=0$, one has for any $\alpha \in (0,\infty)$ that the quadratic polynomial
\[
\mu\mapsto\Jcal_{\alpha}(\mu u^{\eta})=(1-\mu)^2\norm{u^{\eta}}^2_{L^2(\Omega)}+\mu^2\alpha \Rcal(u^{\eta}),
\]
is not minimized at $\mu=0$ or $\mu=1$ because the derivative with respect to $\mu$ does not vanish there.  Hence, $$w^{(\alpha)} \notin\{0, u^{\eta}\}.$$  If we now take $u^c=0$ and suppose additionally that $u^{\eta}$ has zero mean value, then $\Ical(\alpha)>0$ for all $\alpha \in (0,\infty)$, while clearly $\overline{\Ical}(\infty)=0$, that is, the minimum of $\overline{\Ical}$ is only attained at the boundary point $\alpha=\infty$. Similarly, for $u^c=u^{\eta}$, the unique minimizer of $\overline{\Ical}$ is $\alpha=0$.
\end{remark}

\section{Optimal integrability exponents}\label{sec:4}

Here, we study  the optimization of an integrability parameter, $p$, for a fixed nonlocal regularizer. Our motivation comes from the appearance of different $L^p$-norms in image processing, such as in quadratic,  $TV$, and Lipschitz regularization \cite[Section~4]{PCBC10}. We focus on the parameter range $\Lambda= [1,\infty)$ 
with closure $\overline \Lambda = [1, \infty]$, paying particular attention to the structural change occurring at $p= \infty$. 

Let \(\Omega\subset \RR^n\) be a bounded Lipschitz domain and consider a function $f:\Omega \times \Omega\times \R\times \R\to [0, \infty)$ that is Carath\'eodory, i.e., measurable in the first two and continuous with respect to the last two variables, and
that satisfies the following bounds  and  convexity condition: 
 \begin{itemize} 
\item[(H$1_p$)] 
There exist $M, \delta >0$ and $\beta \in [0,1]$ such that for all $\xi,\zeta \in \R$,
we have\[
 f(x,y,\xi,\zeta) \leq M\left(\frac{\abs{\xi-\zeta}}{\abs{x-y}^{\beta}}+\abs{\xi}+\abs{\zeta}+1\right) \quad\text{for a.e.~$x, y\in \Omega$,}
\]
and
\[
M^{-1}\frac{\abs{\xi-\zeta}}{\abs{x-y}^{\beta}}-M \leq f(x,y,\xi,\zeta) \quad \text{for a.e.~$x, y\in \Omega$ with $\abs{x-y}<\delta$.}
\]

\item[(H$2_p$)] $f$ is separately convex in the second two variables, i.e., $f(x, y, \cdot, \zeta)$ and $f(x, y, \xi, \cdot)$ are convex for a.e.~$x,y \in \Omega$ and every $\xi, \zeta\in \R^n$. 
\end{itemize}

In this setting, we take  $p\in [1, \infty)$ and consider the regularization term $\Rcal_p:L^2(\Omega)\to [0, \infty]$  defined by
\begin{align}\label{eq:Rp}
\mathcal R_p(u) := 
\bigg(\frac{1}{|\Omega\times \Omega|}\int_\Omega\int_\Omega f^p(x,y, u(x), u(y)) \,\dx\,\dy \bigg)^{1/p}.
\end{align}

\begin{remark}\label{rem:p}
a) Since the regularizer $\Rcal_p$ is invariant under symmetrization, one can assume without loss of generality that $f$ is symmetric in both pairs of variables, i.e., $f(x, y, \xi,\zeta) = f(y, x, \zeta, \xi)$ and $f(x, y, \xi,\zeta) = f(x, y, \xi, \zeta)$ for all $x, y\in \Omega$ and $\xi, \zeta\in \R$. \smallskip

b) Let $p$, $q\in [1, \infty)$ with $p>q$.  H\"older's inequality then yields for every $u\in {\rm Dom\,} \Rcal_p=\{u\in L^2(\Omega): \Rcal_p(u)<\infty\}$ that
\begin{align*}
\Bigl(\int_\Omega\int_\Omega f^{p}(x,y, u(x), u(y))\,\dx\, \dy\Bigr)^{1/p} \geq |\Omega\times\Omega|^{\frac{q-p}{pq}} \Bigl(\int_\Omega\int_\Omega f^q(x,y, u(x), u(y))\, \dx\, \dy \Bigr)^{1/q},
\end{align*}
which translates into $\Rcal_p(u)\geq 
 \Rcal_q(u)$; in particular, ${\rm Dom\,} \Rcal_p \subset {\rm Dom\,} \Rcal_q$.
\end{remark}

 A basic example of  a symmetric Carath\'eodory function $f$ satisfting (H$1_p$) with $\beta=0$ and (H$2_p$)  is 
\begin{align*}
f(x, y, \xi, \zeta) = a(x-y) |\xi-\zeta|\quad \text{for  $x$, $y\in \Omega$ and $\xi$, $\zeta\in \R$},
\end{align*} 
where $a \in L^{\infty}(\R^n)$ is an even function such that \(\essinf_{\R^n} a >0\). 
Another example of such a function \(f\) with $\beta=1$ in (H$1_p$) is \[
f(x, y, \xi, \zeta) =b\frac{|\xi-\zeta|}{\abs{x-y}}\quad \text{for  $x$, $y\in \Omega$ and $\xi$, $\zeta\in \R$,}
\] 
with $b>0$; note that for the  \(p>n\) case, the corresponding regularizer $\Rcal_p$ is, up to a multiplicative constant, the Gagliardo semi-norm of the fractional Sobolev space $W^{1-\frac{n}{p},p}(\Omega)$. \smallskip

Before showing how the framework of Section~\ref{sec:general} can be applied here, let us first collect and discuss a few properties of the regularizers $\Rcal_p$ with $p\in [1, \infty)$: 

\begin{itemize}
\item[$(i)$] \textit{Upper and lower bounds.} As a consequence of the coercivity bound on the double-integrand $f$ in (H$1_p$), one can deduce constants $C,\, c>0$, depending on $n$, $p$, $\Omega$, $M$, $\delta$, and $\beta$, such that  
\begin{align}\label{est_Lp<Rp}
\norm{u}_{L^p(\Omega)} \leq C\bigl(\Rcal_p(u) + \norm{u}_{L^2(\Omega)}+1\bigr) 
\end{align}
and \begin{align}\label{est_upbeta<Rp}
[u]_{p, \beta}\leq c\bigl(\Rcal_p(u) + \norm{u}_{L^p(\Omega)}+1\bigr)
\end{align}
for all $u\in L^{2}(\Omega)$, where
\[
[u]_{p, \beta}:=\left(\int_\Omega\int_\Omega \frac{|u(x)-u(y)|^p}{|x-y|^{\beta p}}\,\dd{x}\,\dd{y}\right)^{1/p}.
\]
In fact, for \eqref{est_Lp<Rp}, we use the nonlocal Poincar\'{e} inequality in \cite[Proposition~4.2]{BeMC14}, which also holds for $u \in L^2(\Omega)$ via a truncation argument.
From the upper bound in (H$1_p$), we conclude for all $u\in L^2(\Omega)$ that 
\begin{align}\label{est_Rp<upbeta}
\Rcal_p(u)\leq  C\bigl([u]_{p,\beta} + \norm{u}_{L^p(\Omega)}+1\bigr)
\end{align}
with a constant $C=C(p, \Omega, M)>0$.
\smallskip

\item[$(ii)$] \textit{Characterization of the domain.} By combining~\eqref{est_Lp<Rp} and~\eqref{est_upbeta<Rp} with \eqref{est_Rp<upbeta}, it holds for any $p\in [1, \infty)$ that
\begin{align*}
\text{Dom\,} \Rcal_p=\bigl\{u\in L^p(\Omega)\cap L^2(\Omega): [u]_{p, \beta} 
<\infty\bigr\}.
\end{align*}
In particular, $C_c^\infty(\R^n) \subset \text{Dom\,}\Rcal_p$, 
where the functions in $C^\infty_c(\R^n)$  
are implicitly restricted to $\Omega$. 
We observe that for $\beta p > n$, the quantity $[u]_{p,\beta}$ corresponds to the Gagliardo semi-norm of the fractional Sobolev space $W^{\beta-\frac{n}{p},p}(\Omega)$ (cf.~e.g.~\cite{DPV12}),  and so 
\begin{align}\label{chardom}
\text{Dom\,}\Rcal_p=W^{\beta-\frac{n}{p}, p}(\Omega) \cap L^2(\Omega). 
\end{align}
If $\beta p < n$, a simple computation shows that $[u]_{p,\beta}<\infty$ for all $u \in L^p(\Omega)$, which implies $\text{Dom\,}\Rcal_p = L^p(\Omega)\cap L^2(\Omega)$. \smallskip

\item[$(iii)$] \textit{Extension property.}  For any $u \in {\rm Dom\,}\Rcal_p$, there is a $\bar{u} \in L^p(\R^n)\cap L^2(\R^n)$ with compact support inside some bounded open set \(\Omega'\) with $\Omega \subset \Omega' \subset \R^n$ satisfying $\bar{u} = u$ on $\Omega$ and 
\begin{align}\label{extension'}
\int_{\Omega'}\int_{\Omega'} \frac{\abs{\bar{u}(x)-\bar{u}(y)}^p}{\abs{x-y}^{\beta p}}\,\dd{x}\,\dd{y} <\infty.
\end{align}
Indeed, if \(\beta>\frac{n}{p}\), 
this follows directly from well-established extension results for fractional Sobolev spaces on $\Omega$ to those on $\R^n$ (cf.~\cite[Theorem~5.4]{DPV12}), considering~\eqref{chardom}. 
If $1\leq \beta p \leq n$, the map $x \mapsto \abs{x-y}^{-\beta p}$ is no longer integrable at infinity, yet, minor modifications to the arguments in \cite[Section~5]{DPV12} allow us to still deduce~\eqref{extension'}. \smallskip

\item[$(iv)$] \textit{Smooth approximation.} 
 For every $u\in \text{Dom\,}\Rcal_p$, there exists a sequence $(u_{l})_l \subset C_c^{\infty}(\R^n)$ such that $u_l \to u$ in  $L^p(\Omega)$ and $\lim_{l\to \infty}\Rcal_p(u_l) = \Rcal_p(u)$ as $l \to \infty$. 
 
For the proof of this statement, let $\bar u$ be an extension of $u$ as in $(ii)$. 
 We define $u_l = \ffi_{1/l}*\bar{u} \in C_c^{\infty}(\R^n)$ for $l \in \N$ with $(\ffi_{\varepsilon})_{\varepsilon>0}$ a family of smooth standard mollifiers  satisfying $0\leq \ffi_\varepsilon\leq 1$ and  $\int_{\R^n} \ffi_\varepsilon\,
\dd{x}=1$, and  whose support lies in the ball centered at the origin and with radius $\varepsilon>0$,  $\supp \ffi_\varepsilon \subset B_\varepsilon(0)\subset \R^n$. Then,  $u_l \to u$ in $L^p(\Omega)$ and $u_l\to u$ pointwise a.e.~in $\Omega$ as $l \to \infty$. 
 To show that Lebesgue's dominated convergence theorem can be applied, we use the upper bound in (H$1_p$) to derive the following estimate for any $l\in \N$:
 \begin{align}\label{est_fp}
f^p(x, y, u_l(x), u_l(y)) &\leq 4^{p-1}M^p\,\left(\frac{|u_l(x)-u_l(y)|^p}{|x-y|^{\beta p}}+ |u_l(x)|^p + |u_l(y)|^p+1\right)  
\end{align}
for a.e. $(x,y)\in \Omega\times \Omega$. 
By Jensen's inequality and Fubini's theorem, 
\begin{align*} 
[u_l]_{p, \beta}^p& \leq \int_{B_{1/l}(0)}\ffi_{1/l}(z)\int_{\Omega}\int_{\Omega} \frac{\abs{\bar u(x-z)-\bar u(y-z)}^p}{\abs{x-y}^{\beta p}}\,\dd{x}\,\dd{y}\,\dd{z} \\ 
&\leq  \int_{\Omega_{1/l}}\int_{\Omega_{1/l}}\frac{\abs{\bar u(x)-\bar u(y)}^p}{\abs{x-y}^{\beta p}}\,\dd{x}\,\dd{y}<\infty,
\end{align*}
with $\Omega_{1/l} = \{ x \in \R^n\,:\, d(x,\Omega) < 1/l\}$; thus,  $\limsup_{l\to \infty} [u_l]^p_{p, \beta}\leq [u]_{p, \beta}^p$.
Conversely,  the a.e.~pointwise convergence of the mollified sequence gives $\liminf_{l \to \infty} [u_l]^p_{p, \beta} \geq [u]^p_{p, \beta}$ by Fatou's lemma.  Along with the $L^p$-convergence of $(u_l)_l$, the upper bound in~\eqref{est_fp} is thus a converging sequence in $L^1(\Omega\times \Omega)$. This concludes the proof of $\lim_{l \to \infty} \Rcal_p(u_l)= \Rcal_p(u)$. \smallskip

\item[$(v)$]   \textit{Weak lower semicontinuity.} The regularizer $\Rcal_p$ is $L^2$-weak lower semicontinuous.
This is an immediate consequence of the nonnegativity of $f$ and (H$2_p$), see e.g.~\cite[Theorem~2.5]{Ped16} or~\cite{Mun09}; more generally, we refer to~\cite{BeMC18, Elb11, Ped97} for a discussion on sufficient (and necessary) conditions for the weak lower semicontinuity of inhomogeneous double-integral functionals. 
\end{itemize}
\smallskip

Observe that $(ii)$ and $(v)$ imply in particular that the hypothesis \eqref{eq:assumptions} from Section~\ref{sec:general} is fulfilled.

Given a collection of noisy images $u^{\eta}\in L^2(\Omega;\R^N)$ and $p\in [1, \infty)$, we set, for each $j\in\{1,\cdots, N\}$, 
\begin{align*}
\Jcal_{p, j}(u) := \|u-u^{\eta}_j\|^2_{L^2(\Omega)}+ \mathcal R_{p}(u) \qquad\text{for $u\in L^2(\Omega)$,}
\end{align*}
with $K_{p,j}:=\argmin \Jcal_{p,j} \not = \emptyset$ since \eqref{eq:assumptions} is satisfied. As in \eqref{training}, we define $\Ical:[1, \infty)\to
[0, \infty)$ by
  \begin{align*}
  \Ical(p) = \inf_{w \in K_p}\|w-u^{c}\|^2_{L^2(\Omega;\R^N)}
\enspace \text{ for  \(p\in[1,\infty)\)}, \quad
  \end{align*}
where $K_p = K_{p,1} \times K_{p,2} \times \cdots \times K_{p,N}$. Next, we prove the Mosco-convergence result that will provide us with an extension of $\Ical$ to $\overline{\Lambda}=[0, \infty]$. It is an $L^p$-approximation statement in the present nonlocal setting, which can be obtained from a modification of the arguments by Champion, De Pascale, \& Prinari \cite{ChPaPr04} in the local case, 
 and those by Kreisbeck, Ritorto, \& Zappale \cite[Theorem~1.3]{KRZ22}, where the case of homogeneous double-integrands is studied.

\begin{proposition}[Mosco-convergence of the regularizers]\label{prop:Mosco_p}
Let $\Lambda=[1,\infty)$, $\Rcal_p$ for $p\in [1, \infty)$ as in~\eqref{eq:Rp}, and suppose that {\rm (H$1_p$)--(H$3_p$)} are satisfied. Then,
for $p \in \overline{\Lambda}=[1,\infty]$,\begin{align}\label{Mosco_result_p}
\overline{\Rcal}_{p}:=\text{\rm Mosc}(L^2)\text{-}\lim_{p'\to p}\Rcal_{p'}=\begin{cases}
\Rcal_p &\text{if} \ p \in [1,\infty),\\
\Rcal_{\infty} &\text{if} \ p=\infty,
\end{cases} \quad \text{}
\end{align}
with $\Rcal_\infty:L^2(\Omega)\to [0, \infty]$ given by
\begin{align*}
\mathcal R_\infty (u):= \esssup_{(x,y)\in \Omega\times \Omega} f(x,y, u(x), u(y)). 
\end{align*}
\end{proposition}

\begin{proof}
To show  \eqref{Mosco_result_p}, it suffices to show that for every sequence $(p_k)_k \subset [1,\infty)$  converging to \(p \in [1,\infty]\), \eqref{Mosco_result_p} holds with \(p'\) replaced by \(p_{k}\). We divide the proof into two cases. \smallskip

\textit{Case~1: $p \in [1,\infty)$.} 
For the recovery sequence, consider $u \in \mathrm{Dom}\,\Rcal_p$ and take $(u_l)_{l}\subset C_c^\infty(\R)$ as in $(iv)$, satisfying
  $u_l \to u$ in $L^p(\Omega)$ and $\Rcal_p(u_l) \to \Rcal_p(u)$ as $l \to \infty$.  
 In view of~$(ii)$, we know that $(u_l)_l$ is contained in ${\rm Dom\,}\Rcal_{p}$ and ${\rm Dom\,}\Rcal_{p_k}$ for all $k\in \N$, and we conclude via Lebesgue's dominated convergence theorem that
\[
\lim_{k \to \infty} \Rcal_{p_k}(u_l) = \Rcal_p(u_l)
\]
for every $l \in \N$. Hence,
\begin{align*}
\lim_{l\to \infty}\lim_{k\to \infty} \Rcal_{p_k}(u_l) = \lim_{l\to \infty}\Rcal_{p}(u_l) = \Rcal_{p}(u),
\end{align*} 
so that one can find a recovery sequence by extracting an appropriate diagonal sequence. 

To prove the lower bound, let $u_k\weakly u$ in $L^2(\Omega)$ be such that $\lim_{k\to \infty} \Rcal_{p_k}(u_k)=\liminf_{k\to \infty} \Rcal_{p_k}(u_k)<\infty$, and 
fix $s\in (1, p)$ (or $s=1$ if $p=1$). Observe that $p_k\geq s$ for all $k$ sufficiently large because  $p_k\to p$ for $k\to \infty$. Then, Remark~\ref{rem:p}\,b) and the weak lower semicontinuity of $\Rcal_s$ according to $(v)$ imply that
\begin{align*}
\lim_{k\to \infty} \Rcal_{p_k}(u_k) \geq \liminf_{k\to \infty} \Rcal_{s}(u_k)\geq \Rcal_s(u).
\end{align*} 
If $s=p=1$ the argument is complete, whereas in the case $p>1$, an additional application of Fatou's lemma shows $\liminf_{s \nearrow p} \Rcal_s(u) \geq \Rcal_p(u)$, giving rise to the desired liminf inequality. \smallskip

\textit{Case~2: $p=\infty$.} That constant sequences serve as recovery sequences results from the observation that $\Rcal_{p_k}(u) \to \Rcal_{\infty}(u)$ as $k \to \infty$ for all $u \in \mathrm{Dom}\,\Rcal_{\infty}$. The latter is an immediate consequence of classical $L^p$-approximation, i.e., the well-known fact that $\lim_{p\to \infty}\norm{v}_{L^p(V)} = \norm{v}_{L^\infty(V)}=\esssup_{x\in V}|v(x)|$ for all $v\in L^\infty(V)$ with $V\subset \R^m$ open and bounded.

To prove the lower bound, we argue via Young measure theory (see, e.g.,~\cite{Ped97,FoL07} for a general
introduction). Let $u_k\weakly u$ in $L^2(\Omega)$, and denote by $\nu=\{\nu_x\}_{x\in\Omega}$  the Young measure generated by a (non-relabeled) subsequence of $(u_k)_k$. The barycenter of $[\nu_x]:= \int_{\R} \xi \, \dd{\nu_x}(\xi)$ then coincides with $u(x)$ for a.e.~$x \in \Omega$.
Without loss of generality, one can suppose that $\infty>\liminf_{k\to \infty}\Rcal_{p_k}(u_k)=\lim_{k\to \infty}\Rcal_{p_k}(u_k)$. 
Recalling Remark~\ref{rem:p}\,b), we have that
\begin{align}\label{356}
\lim_{k\to \infty}\Rcal_{p_k}(u_k) \geq \liminf_{q\to \infty} \liminf_{k\to \infty}\Rcal_{q}(u_k). 
\end{align}

 On the other hand,  with the nonlocal field $v_u$ associated with some $u:\Omega\to \R$ defined by 
\begin{align*}
v_u(x,y):=(u(x), u(y))\qquad\text{ for $(x,y)\in \Omega\times \Omega$},
\end{align*}
the statement of~\cite[Proposition~2.3]{Ped97} allows us to extract a subsequence $(v_{u_k})_k$ that generates the Young measure $\{\nu_x\otimes \nu_y\}_{(x,y)\in \Omega\times \Omega}$. Hence, a standard result on Young measure lower semicontinuity (see e.g.~\cite[Section~8.1]{FoL07}) yields
\begin{align*}
\liminf_{k\to \infty} \Rcal_{q}(u_k) & \geq\Bigl(\frac{1}{|\Omega\times \Omega|} \int_\Omega\int_\Omega\int_{\R }\int_{\R } f^q(x,y,\xi, \zeta) \,\dd \nu_x(\xi)\,\dd \nu_{y}(\zeta)\,\dd{x}\,\dd{y}\Bigr)^{1/q}. 
\end{align*}

Letting $q\to \infty$, we use classical $L^p$-approximation results and the Jensen's type inequality for separately convex functions in~\cite[Lemma~3.5]{KrZ20} to conclude that 
\begin{align*}
\liminf_{q\to \infty}\liminf_{k\to \infty} \Rcal_{q}(u_k) & 
\geq \esssup_{(x, y)\in \Omega\times \Omega} (\nu_x\otimes \nu_y)\text{-}\esssup_{(\xi, \zeta)\in \R \times \R } f(x,y, \xi, \zeta) \\ &\geq \esssup_{(x, y)\in \Omega\times \Omega} f(x,y, [\nu_x], [\nu_y]) \\ &= \esssup_{(x, y)\in \Omega\times \Omega} f(x,y, u(x), u(y)) = \Rcal_\infty(u).
\end{align*}
 Finally, the lower bound follows from the previous estimate and~\eqref{356}.
\end{proof}

The above result implies that the reconstruction functional for $p=\infty$ and $j \in \{1,\cdots,N\}$ is given by
\begin{align*}
\overline{\Jcal}_{\infty, j}(u) := \|u-u^{\eta}_j\|^2_{L^2(\Omega)}+ \mathcal R_{\infty}(u) \qquad\text{for $u\in L^2(\Omega)$.}
\end{align*}
Under the additional convexity condition on the given function $f:\Omega\times \Omega\times\R^n\times \R^n\to \R$ that 
\begin{itemize}
\item[(H$3_p$)] $f$ is (jointly) level convex in its last two variables,
\end{itemize}
where level convexity means convexity of the sub-level sets of the function, the supremal functional $\Rcal_\infty$ also becomes level convex. In combination with the strict convexity of the fidelity term, the reconstruction functional $\overline{\Jcal}_{\infty,j}$ then admits a unique minimizer. Since level convexity is weaker than convexity, we do not necessarily have that $\Jcal_{p,j}$ for $p \in [1,\infty)$ is (level) convex, and it may have multiple minimizers.

If we suppose that $f$ fulfills (H$1_p$)--(H$3_p$), then Theorem~\ref{theo:relax2} and Proposition~\ref{prop:Mosco_p} imply that the extension $\overline{\Ical}:[1,\infty] \to [0,\infty]$ is given by
\begin{align*}
\overline{\Ical}(p) = \begin{cases}
\Ical(p) & \text{for $p\in [1, \infty)$,}\\
\norm{w^{(\infty)}-u^c}^2_{L^2(\Omega;\R^N)} & \text{for $p=\infty$,}
\end{cases} \quad 
\end{align*}
\text{for $p \in [1,\infty]$,}
where $w^{(\infty)}$ denotes the unique componentwise minimizer of $\overline{\Jcal}_{\infty}$. 
In particular, the hypothesis $(ii)$ of Theorem~\ref{theo:relax2} is satisfied, which shows that $\overline{\Ical}$ is the relaxation of $\Ical$ and, thus, admits a minimizer $\bar{p} \in \overline{\Lambda}=[1,\infty]$.

 We conclude this section with a discussion of examples when optimal values of the integrability exponents are obtained in the interior of the original interval $\Lambda$ or at its boundary, respectively. In one case, the presence of noise causes $\Rcal_{\infty}$ to penalize $u^c$ more than $u^{\eta}$, while $\Rcal_q$ for some $q \in [1,\infty)$ prefers the clean image. This entails that the optimal parameter is attained in $\Lambda=[1,\infty)$. In the second case instead, the reconstruction functional for $p=\infty$ gives back the exact clean image and outperforms the reconstruction functionals for other parameter values. 
\begin{example}\label{ex:exponent}
a) Let $f=\alpha \widehat{f}:\Omega\times \Omega\times \R^n\times \R^n\to \R$, for some $\alpha >0$ to be specified later, be a double-integrand satisfying (H$1_p$),  (jointly) convex in the last two variables, and vanishing exactly on $\{(x,y,\xi,\xi)\,:\, x,y \in \Omega, \ \xi \in \R\}$. Following~\eqref{eq:Rp}, we set $$\Rcal_p(u)= \alpha \bigg(\frac{1}{|\Omega\times \Omega|}\int_\Omega\int_\Omega \widehat{f}^p(x,y, u(x), u(y)) \,\dd{x}\,\dd{y} \bigg)^{1/p} =: \alpha \widehat{\Rcal}_p(u)$$ for $u\in L^2(\Omega)$ and $p\in [1, \infty)$. 

We further introduce the following two conditions on the given data $u^\eta, u^c\in L^2(\Omega;\R^N)$: \smallskip
 \begin{itemize} 
\item[(H$4_p$)] $\sum_{j=1}^N\Rcal_q(u_j^c)<\sum_{j=1}^N\Rcal_q(u_j^{\eta}) \ \text{for some $q \in [1,\infty)$};$\medskip
\item[(H$5_p$)] $\sum_{j=1}^N\Rcal_\infty(2u_j^{\eta} -u_j^c) < \sum_{j=1}^N\Rcal_\infty(u^{\eta}_j).$
\end{itemize} 
\color{black}
\smallskip
By applying Lemma~\ref{lemma:not-min}\,$(i)$ from the previous section with $\Rcal=\widehat{\Rcal}_q$ --- the conditions (H$1_\alpha$), (H$2_\alpha$), and (H$3_\alpha$) are immediate to verify in view of $(ii)$, $(v),$ and (H$4_p$) --- we can then deduce for small enough $\alpha$ that $\overline{\Ical}(q) < \norm{u^{\eta}-u^{c}}^2_{L^2(\Omega;\R^N)}$. \color{black} On the other hand, due to (H$5_p$), the same lemma can be applied to $\Rcal=\widehat{\Rcal}_\infty$ with $\widehat{\Rcal}_\infty(u)=\esssup_{(x, y)\in \Omega\times \Omega}  \widehat f(x, y, u(x),u(y))$ for $u\in L^2(\Omega)$ to find 
\begin{equation}
\label{eq:Iinfty-worse}
\|w^{(\infty)}-(2u^{\eta}-u^c)\|_{L^2(\Omega;\R^N)}^2<\|u^{\eta}-u^{c}\|_{L^2(\Omega;\R^N)}^2,
\end{equation}
provided $\alpha$ is sufficiently small. 
\color{black} The reverse triangle inequality then yields 
\begin{align*}
&\overline{\Ical}(\infty) \geq \left(\norm{w^{(\infty)}-(2u^{\eta}-u^c)}^2_{L^2(\Omega;\R^N)}-2\norm{u^{\eta}-u^{c}}^2_{L^2(\Omega;\R^N)}\right)^2\\
&\qquad> \norm{w^{(\infty)}-(2u^{\eta}-u^c)}^2_{L^2(\Omega;\R^N)}>
 \norm{u^{\eta}-u^{c}}^2_{L^2(\Omega;\R^N)} > \overline{\Ical}(q),
\end{align*}
where in the second and third inequality we have used \eqref{eq:Iinfty-worse}.
This  proves that the optimal parameter is attained inside $[1,\infty)$, and, therefore, is also a minimizer of~$\Ical$. \smallskip

\color{black}
b) We illustrate a) with a specific example. Consider $\Omega = (0,1)$ and let $\widehat f(x,y,\xi,\zeta)= \abs{\xi-\zeta}/\abs{x-y}$ for $x$, $y\in \Omega$ and $\xi$, $\zeta\in \R^n$.  
\color{black} This leads then to the difference quotient regularizers
\begin{align}\label{Rcalp}
\Rcal_p(u) = \alpha \left(\int_0^1\int_0^1 \frac{\abs{u(x)-u(y)}^p}{\abs{x-y}^p}\, \dx \, \dy\right)^{1/p} =: \alpha \widehat{\Rcal}_p(u)
\end{align}
and
\begin{align}\label{Rcalinfty}
\Rcal_{\infty}(u) = \alpha \esssup_{(x,y) \in (0,1)^2}\frac{\abs{u(x)-u(y)}}{\abs{x-y}} = \alpha \mathrm{Lip}(u),
\end{align}
with $\mathrm{Lip}(u)$ denoting the Lipschitz constant of (a representative of) $u$, which could be infinite.

With the sawtooth function
\(v: [0,1] \to \R\) defined by\[
 v(x) = \begin{cases}
x \quad&\text{for $0\leq x\leq 1/4$},\\
-x+1/2 \quad&\text{for $1/4 <x \leq 3/4$},\\
x-1 \quad&\text{for $3/4 < x \leq 1$}, 
\end{cases}
\]
 we take a single clean and noisy image given by
\[
u^{c}(x) = \begin{cases}
0 &\text{for $0< x\leq 1/3$},\\
10v(3x-1) &\text{for $1/3 <x \leq 2/3$}\\
0 &\text{for $2/3 < x < 1$}.
\end{cases}
\quad
\text{and}
\quad
u^{\eta}(x) = \begin{cases}
v(3x) &\text{for $0< x\leq 1/3$},\\
(10-\varepsilon)v(3x-1) &\text{for $1/3 <x \leq 2/3$},\\
v(3x-2) &\text{for $2/3 < x < 1$},
\end{cases}
\]
respectively, where $\varepsilon >0$ is small; see Figure~\ref{fig:exponent}.
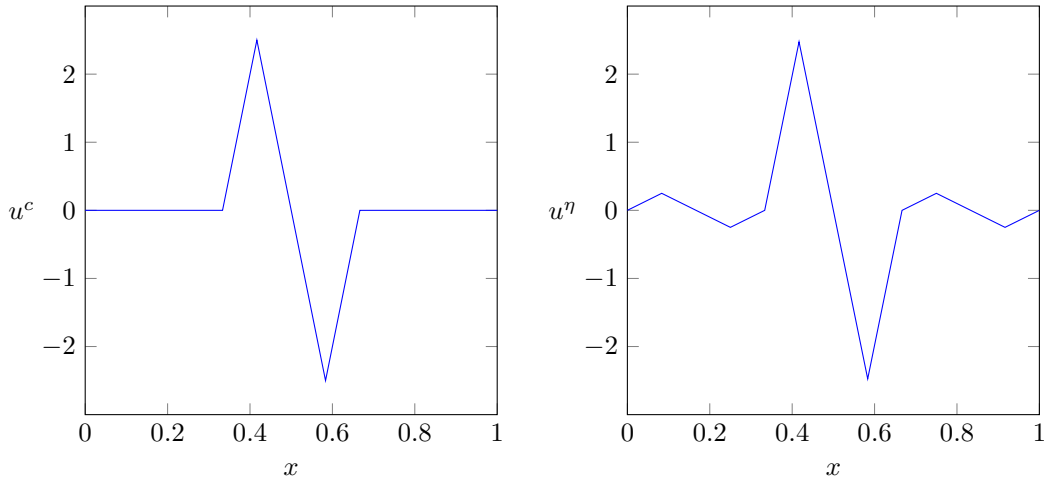
\begin{figure}
\centering
\begin{tabular}{c c} 
\begin{tikzpicture}[
declare function={
v(\x)= and(\x>0, \x<=1/4) * (\x)   +
and(\x>1/4, \x<=3/4) * (-\x+1/2)     +
and(\x>3/4, \x<=1) * (\x-1);
uc(\x) = and(\x>1/3, \x<=2/3) * (10*v(3*\x-1));  
}
]
\begin{axis}[
width = 7cm,
height = 7cm,
xlabel = {$x$},
ylabel = {$u^{c}$},
ylabel near ticks,
ylabel style={at={(axis description cs:-0.1,.5)},rotate=-90},
ytick = {-2,-1,0,1,2},
xmin = 0,
xmax = 1,
ymin = -3,
ymax = 3
]
\addplot[blue, domain=0:1]{uc(x)};
\end{axis}
\end{tikzpicture}
&
\begin{tikzpicture}[
declare function={
v(\x)= and(\x>0, \x<=1/4) * (\x)   +
and(\x>1/4, \x<=3/4) * (-\x+1/2)     +
and(\x>3/4, \x<=1) * (\x-1);
uet(\x) = (\x <= 1/3) * (v(3*\x)) +
and(\x>1/3, \x <=2/3) * (9.9*v(3*\x-1)) +
(\x>2/3) * v(3*\x-2);  
}
]
\begin{axis}[
width = 7cm,
height = 7cm,
xlabel = {$x$},
ylabel = {$u^{\eta}$},
ylabel near ticks,
ylabel style={at={(axis description cs:-0.1,.5)},rotate=-90},
xmin = 0,
xmax = 1,
ymin = -3,
ymax = 3,
ytick = {-2,-1,0,1,2}
]
\addplot[blue, domain=0:1]{uet(x)};
\end{axis}
\end{tikzpicture}
\end{tabular}
\caption{The graphs of the functions $u^{c}$ and $u^{\eta}$ from Example~\ref{ex:exponent}~a) with $\varepsilon =0.1$.}
\label{fig:exponent}
\end{figure}
We observe that  $u^c$ is constant near the boundaries and only slightly steeper than $u^{\eta}$ in the middle of the domain. Numerical calculations show that for small $\varepsilon$, such as~$\varepsilon=0.1$, the estimate $\Rcal_2(u^c)<\Rcal_2(u^{\eta})$, and hence (H$4_p$) with $q=2$, holds; moreover, (H$5_p$) holds since the clean image has a higher Lipschitz constant than the noisy image in the sense that
\[
\mathrm{Lip}(2u^{\eta} -u^c) = 30-6\varepsilon < 30-3\varepsilon = \mathrm{Lip}(u^{\eta}).
\]
Therefore, we find that for $\alpha>0$ small enough, the optimal parameter lies inside $\Lambda=[1,\infty)$.
\smallskip

c) If we work with the same regularizers as in b), there are reasonable images for which the Lipschitz regularizer~in~\eqref{Rcalinfty} performs better than the other regularizers~in \eqref{Rcalp}. Let us consider with $\alpha >0$ chosen as in b), the images 
\begin{align*}
u^{c}(x) = x-1/2\quad \text{ and}\quad  u^{\eta} = (1+6\alpha)u^{c}. 
\end{align*}
Since $u^{\eta}$ is affine, we can show that the reconstruction with the Lipschitz regularizer is also an affine function. Indeed, for every other function, one can find an affine function with at most the same Lipschitz constant without increasing the distance to $u^{\eta}$ anywhere. This, in combination with the fact that the images are odd functions with respect to $x=1/2$, shows that $w^{(\infty)}$ is of the form $w^{(\infty)}(x) = \gamma (x-1/2)=\gamma u^c$ with $\gamma \geq 0$. Due to the optimality of $w^{(\infty)}$, the constant $\gamma$ has to minimize the quantity
\[
\norm{\gamma u^c- u^\eta}_{L^2((0,1))}^2+\alpha{\rm Lip}(\gamma u^c)=\frac{1}{12}(\gamma - (1+6\alpha))^2 + \alpha \gamma,
\]
which yields $\gamma = 1$. \color{black} Hence, $w^{(\infty)}$ coincides with the clean image and therefore $\overline{\Ical}(\infty)=0$, which implies that $p = \infty$ is the optimal parameter in this case.
\end{example}
\section{Varying the amount of nonlocality}\label{sec:amountnonlocality}

Next, we study
two classes
of nonlocal regularizers,  $\Rcal_{\delta}$ with 
$\delta \in \Lambda :=(0,\infty)$, considered
by Brezis \& Nguyen \cite{BrNg18} and 
Aubert \& Kornprobst \cite{AuKo09}, respectively, in the context of
 image processing. In both cases, we aim at optimizing the parameter \(\delta\) that encodes the amount
of nonlocality in the problem. We mention further that  both families  of functionals recover the classical $TV$-reconstruction
model in   the limit $\delta
\to 0$, cf.~\cite{BrNg18,AuKo09}.

To set the stage for our analysis, consider training data $(u^c,u^{\eta}) \in L^2(\Omega;\R^N) \times L^2(\Omega;\R^N)$ and the reconstruction functionals $\Jcal_{\delta,j}:L^2(\Omega) \to [0,\infty]$ with $\delta \in \Lambda$ and $j \in \{1,2,\dots,N\}$  given by
\[
\Jcal_{\delta,j}(u) = \norm{u-u^{\eta}_j}_{L^2(\Omega)}^2 + \Rcal_{\delta}(u).
\]
After showing that the sets 
\begin{equation}
\label{eq:kdeltaj}
K_{\delta,j}=\argmin_{u \in L^2(\Omega)} \Jcal_{\delta,j}(u).
\end{equation}
are non-empty for each of the two choices of the regularizers $\Rcal_{\delta}$, the upper-level functional from \eqref{training} in Section~\ref{sec:general} becomes
\begin{align}\label{Ical_sec5}
\Ical:(0,\infty) \to [0,\infty), \quad \Ical(\delta) = \inf_{w \in K_\delta} \norm{w-u^c}_{L^2(\Omega;\R^N)}^2
\end{align}
with $K_\delta = K_{\delta,1}\times K_{\delta,2} \times \cdots \times K_{\delta,N}$. In order 
to find its extension $\overline{\Ical}$ defined on $\overline{\Lambda}=[0,\infty]$, we determine the Mosco-limits of the regularizers (cf.~\eqref{eq:extI} and Theorem~\ref{theo:relax2}). This is the content of Propositions~\ref{th:BNgammanonlocal} and \ref{th:gammanonlocal} below, which provide the main results of this section.

\subsection{Brezis \& Nguyen setting}
\label{sub:BN}

For every $\delta\in (0,\infty)$ and $u\in L^1(\Omega)$, we consider the regularizers 
\[
\Rcal_\delta(u):=\delta\int_\Omega\int_\Omega \frac{\varphi(|u(x)-u(y)|/\delta)}{|x-y|^{n+1}}\,\dd x\,\dd y,
\]
where, following~\cite{BrNg18}, the function $\varphi:[0, \infty)\to [0, \infty)$ is  assumed to satisfy the following hypotheses: 
\begin{itemize}
\item[(H$1_\delta$)] $\varphi$ is lower semicontinuous in \([0,\infty)\) and continuous in \([0,\infty)\) except at a finite number of points, where it admits left- and right-side limits;\smallskip
\item[(H$2_\delta$)] there exists a constant $a>0$ such that $\ffi(t) \leq \min\{at^2,a\}$ for all $t\in [0,\infty)$; \smallskip
\item[(H$3_\delta$)] $\varphi$ is non-decreasing; \smallskip 
\item[(H$4_\delta$)] it holds that $\displaystyle \gamma_n \int_0^\infty \ffi(t) t^{-2} \, \dd{t} =1$ with $\gamma_n:= \displaystyle \int_{\Sbb^{n-1}} |e\cdot \sigma|\, \dd{\sigma}$ for any $e\in \Sbb^{n-1}$.
\end{itemize}

Note that the assumptions on $\varphi$ imply that the functional $\Rcal_{\delta}$ is never convex. 

\begin{example}\label{egAcal}
Examples of functions $\varphi$ with the properties (H$1_\delta$)--(H$4_\delta$) include suitable normalizations of
\begin{align*}
t\mapsto
\begin{cases}
0 &\text{if } t\leq 1\\
1&\text{if } t> 1
\end{cases}, \qquad 
t\mapsto\begin{cases}
t^2 &\text{if } t\leq 1\\
1&\text{if } t> 1
\end{cases}, \qquad t\mapsto 1-e^{-t^2}
\end{align*}
for $t\geq 0$,
cf.~\cite{BrNg18}.
\end{example}
To guarantee that the functionals $\Rcal_{\delta}$  satisfy a suitable compactness property, see Theorem~\ref{thm:BN-thm1}\,b),  
we must additionally assume that \smallskip
\begin{itemize}\label{eq:positive}
\item[(H$5_\delta$)] $\varphi(t) > 0 \  \text{for all $t>0$}$.
\end{itemize}
\smallskip

 Clearly, the last two functions from Example~\ref{egAcal} satisfy the positivity condition, while the first one does not.
In identifying the Mosco-limits $\overline{\Rcal}_{\delta}$ in each of the three cases \(\delta\in(0,\infty)\), \(\delta=0\), and \(\delta=\infty\), we make repeated use of ~\cite[Theorems~1, 2
and 3]{BrNg18}, which we recall here for the reader's convenience.

\begin{theorem}[cf.~{\cite[Theorems~1--3]{BrNg18}}]\label{thm:BN-thm1}
Let \(\Omega\subset \RR^n\) be a bounded and smooth domain, and let $\varphi$ satisfy {\rm (H$1_\delta$)--(H$4_\delta$)}.

\noindent a) If $(\delta_k)_k\subset (0,\infty)$ is such that \(\delta_k\to0\), then the following statements hold:

\begin{itemize}
\item[$(i)$]

There exists a constant \(K(\ffi) \in (0,1]\), independent of \(\Omega\), such that \((\Rcal_{{\delta_k}})_{k}\) \(\Gamma\)-converges as $k \to \infty$, with respect to
the \(L^1(\Omega)\)-topology, to \(\Rcal_0:L^{1}(\Omega) \to [0,\infty]\) defined for \(u\in L^1(\Omega)\) by
\begin{equation*}
\begin{aligned}
\Rcal_0(u):=\begin{cases}
K(\varphi) |Du|(\Omega) &\text{if } u\in BV(\Omega),\\
\infty &\text{if } u\in L^1(\Omega)\setminus BV(\Omega).
\end{cases}
\end{aligned}
\end{equation*}

\item[$(ii)$] 
If \((u_k)_k\) is a bounded sequence in \(L^1(\Omega)\) with \(\sup_k
\Rcal_{\delta_k}(u_k) <\infty\), then there exist a subsequence \((u_{k_l})_l\)
of \((u_k)_k\)  and a function \(u\in L^1(\Omega)\) such that \(\lim_{l\to \infty} \Vert
u_{k_l} - u\Vert_{L^1(\Omega)}=0\).
\end{itemize}

\noindent b) Suppose that {\rm (H$5_\delta$)} holds in addition to the above conditions, and let \((u_k)_k\) be a bounded sequence in \(L^1(\Omega)\) with \(\sup_k
\Rcal_{\delta}(u_k) <\infty\) for some $\delta>0$. Then, there exists a subsequence \((u_{k_l})_l\)
of \((u_k)_k\)  and a function \(u\in L^1(\Omega)\) such that \(\lim_{l\to \infty} \Vert
u_{k_l} - u\Vert_{L^1(\Omega)}=0\).
\end{theorem}

We point out that if $\varphi$ fulfills (H$1_\delta$)--(H$5_\delta$), then (H) in Section~\ref{sec:general} holds and the sets $K_{\delta,j}$ defined in \eqref{eq:kdeltaj} are non-empty (cf.~\cite[Corollary~7]{BrNg18}). 
We are now in a position to characterize the asymptotic behavior of the regularizers $\Rcal_{\delta'}$ as $\delta'\to \delta\in \overline{\Lambda}=[0,\infty]$.

\begin{proposition}[Mosco-convergence of regularizers]
\label{th:BNgammanonlocal}
Let $\Lambda=(0,\infty)$ and \(\Omega\subset \RR^n\) be a bounded and smooth domain. Under the assumptions 
 {\rm (H$1_\delta$)--(H$5_\delta$)} on $\varphi: [0, \infty)\to [0, \infty)$, it holds that
\begin{equation}\label{eq:BNgammanonlocal}
\overline{\Rcal}_{\delta}:=\text{\rm Mosc}(L^2)\text{-}\lim_{\delta'\to \delta}\Rcal_{\delta'}=\begin{cases}
\Rcal_\delta &\text{if} \ \delta \in (0,\infty),\\
\Rcal_{0} &\text{if} \ \delta=0,\\
0 &\text{if $\delta=\infty$},
\end{cases} \quad \text{for $\delta \in \overline{\Lambda}=[0,\infty]$.}
\end{equation}
\end{proposition}

\begin{proof}
Considering a sequence $(\delta_k)_k\subset (0,\infty)$ with limit $\delta \in [0,\infty]$, one needs to verify that the Mosco-limit of $(\Rcal_{{\delta}_k})_k$ exist and is given by the right-hand side of \eqref{eq:BNgammanonlocal}. We split the proof into three cases. \smallskip

\textit{Case 1: $\delta =0$.} Let \((u_k)_k\subset
L^2(\Omega)\) and \(u\in
L^2(\Omega)\) be such that \(u_k\weakly  u\) in \(L^2(\Omega)\).
We aim to show that
\begin{equation}\label{eq:liminf0}
\Rcal_0(u) 
\leq \liminf_{k\to\infty} \Rcal_{\delta_k}(u).
\end{equation}
One may thus  assume without loss of generality
that the limit inferior on the right-hand side of \eqref{eq:liminf0}
is finite, and, after extracting a subsequence if necessary, also
\begin{equation*}
\begin{aligned}
\sup_k \Rcal_{\delta_k}(u_k)
<\infty.
\end{aligned}
\end{equation*}
Hence, by Theorem~\ref{thm:BN-thm1}\,a)\,$(ii)$, it follows that
$u_k \to u \text{ in }  L^1(\Omega)$,
which together with Theorem~\ref{thm:BN-thm1}\,a)\,$(i)$ yields \eqref{eq:liminf0}.

To complement this lower bound, we need to obtain for each \(u\in L^2(\Omega)\cap BV(\Omega)\)
 a sequence \((u_k)_k\subset
L^2(\Omega)\) such that  \(u_k\to  u\) in \(L^2(\Omega)\)
and
\begin{equation}\label{eq:liminf02}
\begin{aligned}
\Rcal_0(u)
 \geq \limsup_{k\to\infty} \Rcal_{\delta_k}(u_k). 
\end{aligned}
\end{equation}  
The idea is to suitably truncate a recovery sequence of the $\Gamma$-limit $\Gamma(L^1)$-$\lim_{k\to \infty}\Rcal_{\delta_k}$ from Theorem~\ref{thm:BN-thm1}\,$(i)$. 
 For the details, fix \(l\in\NN\) and consider the truncation function, \(T^l:\RR\to\RR\),
\begin{align*}
T^l(t):=\begin{cases}
l & \hbox{if } t\geq l,\\
t & \hbox{if } -l\leq t\leq l,\\
-l & \hbox{if } t\leq -l.
\end{cases}
\end{align*} 
By Theorem~\ref{thm:BN-thm1}\,$(i)$, there exists a sequence \((v_k)_k\subset L^1(\Omega)\) such that \(v_k\to u\) in \(L^1(\Omega)\) and 
\begin{equation}\label{eq:recseq}
\begin{aligned}
\lim_{k\to\infty} \Rcal_{\delta_k}(v_k)
= K(\ffi)|Du|(\Omega) = \Rcal_0(u).
\end{aligned}
\end{equation}
Choosing a sequence $(l_k)_k \subset \R$ such that $l_k \to \infty$ and $l_k\norm{v_k-u}_{L^1(\Omega)} \to 0$ as $k \to \infty$, we define 
\[
u_k:=T^{l_k} \circ v_k  \in L^\infty(\Omega) \quad \text{for all $k\in \N$}.
\]
Then, an application of H\"{o}lder's inequality shows that
\begin{align*}
\norm{u_k -u}_{L^2(\Omega)} &\leq \norm{u_k-T^{l_k}\circ u}_{L^2(\Omega)}+\norm{T^{l_k} \circ u-u}_{L^2(\Omega)}\\ &\leq \left(2l_k\norm{v_k-u}_{L^1(\Omega)}\right)^{1/2} + \norm{T^{l_k} \circ u-u}_{L^2(\Omega)}  \to 0,
\end{align*}
as $k \to \infty$. Therefore, $u_k \to u$ in $L^2(\Omega)$ and, in view of the monotonicity of $\varphi$ in (H$3_\delta$), we conclude that
\begin{align*}
\limsup_{k\to \infty} \Rcal_{\delta_k}(u_k) 
&=  \limsup_{k\to\infty}
\delta_k\int_{\Omega}\int_{\Omega}
\frac{\varphi(\delta_k^{-1}|T^{l_k} (v_k (x))-T^{l_k} (v_k (y))|)}{|x-y|^{n+1}} \,\dx\,\dy\\
&\leq \lim_{k\to\infty} \delta_k\int_{\Omega}\int_{\Omega}
\frac{\varphi(\delta_k^{-1}| v_k(x)-v_k (y)|)}{|x-y|^{n+1}}
\,\dx\,\dy =\lim_{k\to\infty} \Rcal_{\delta_k}(v_k),
\end{align*}
 \color{black}
which implies \eqref{eq:liminf02}  by \eqref{eq:recseq}. \smallskip

\textit{Case 2: $\delta \in (0,\infty)$.} Consider a sequence \((u_k)_k\subset L^2(\Omega)\) and \(u\in
L^2(\Omega)\) such that \(u_k\weakly  u\) in \(L^2(\Omega)\) and
\begin{equation*}
\begin{aligned}
\sup_k \Rcal_{\delta_k}(u_k)
<\infty.
\end{aligned}
\end{equation*}
 
We start by observing that there
exist \(\bar \delta > 0\) and \(K\in
\NN\) such that for all   \(k\geq K\), we have
$\bar{\delta}/2 \leq \delta_k \leq \bar{\delta}$. Hence, the previous estimate
and~(H$3_\delta$) yield
\begin{equation*}
\begin{aligned}
\sup_{k\geq K} \Rcal_{\bar\delta}(u_k)=
\sup_{k \geq K} \bigg(  \bar \delta\int_{\Omega}\int_{\Omega}
\frac{\varphi({\bar \delta}^{-1}|u_k(x)-u_k(y)|)}{|x-y|^{n+1}}
\,\dx\,\dy\bigg) \leq 2\sup_{k}  \Rcal_{\delta_k}(u_k) <\infty.
\end{aligned}
\end{equation*}
Consequently, in view of Theorem~\ref{thm:BN-thm1}\, b), we may further
assume that 
\begin{equation}\label{eq:liminf+1}
\begin{aligned}
u_k \to u \text{ in }  L^1(\Omega) \qquad \text{ and } \qquad u_k (x) \to u(x)
\text{ for a.e.~$x\in\Omega$}.
\end{aligned}
\end{equation}
Using Fatou's
lemma first, and then \eqref{eq:liminf+1} together with the lower semicontinuity
of \(\ffi\) on \([0,\infty)\), we get
\begin{align*}
\liminf_{k\to \infty}\Rcal_{\delta_k} (u_k) &= \liminf_{k\to\infty} \delta_k\int_{\Omega}\int_{\Omega}
\frac{\varphi(\delta_k^{-1}|u_k(x)-u_k(y)|)}{|x-y|^{n+1}} \,\dx\,\dy\\
 &\geq \,\delta \int_{\Omega}\int_{\Omega} \liminf_{k\to\infty}
\frac{\varphi(\delta_k^{-1}|u_k(x)-u_k(y)|)}{|x-y|^{n+1}} \,\dx\,\dy\\
 &\geq \,\delta \int_{\Omega}\int_{\Omega}
\frac{\varphi(\delta^{-1}|u(x)-u(y)|)}{|x-y|^{n+1}} \,\dx\,\dy= \Rcal_\delta(u),
\end{align*}
which proves the liminf inequality.

\color{black}
For the recovery sequence, fix $u \in L^2(\Omega)$ and take $u_k = \frac{\delta_k}{\delta}u$ for $k \in \N$. Then, $u_k \to u$ in $L^2(\Omega)$ as $k \to \infty$ and
\[
\lim_{k \to \infty} \Rcal_{\delta_k}(u_k) = \lim_{k \to \infty} \frac{\delta_k}{\delta}\Rcal_{\delta}(u) = \Rcal_{\delta}(u),
\]
as desired. \smallskip

\textit{Case 3: $\delta = \infty$.} The lower bound follows immediately by the non-negativity of $\Rcal_{\delta_k}$ for $k \in \N$.  As a recovery sequence for $u \in L^2(\Omega)$, take a sequence $(u_k)_k \subset L^2(\Omega)$ such that $u_k \to u$ in $L^2(\Omega)$ and $\mathrm{Lip}(u_k) \leq \delta_k^{1/4}$, which is possible since $\delta_k \to \infty$ as $k \to \infty$. Then, using (H$2_\delta$), 
\begin{align*}
\Rcal_{\delta_k}(u_k)&= \delta_k\int_{\Omega}\int_{\Omega}
\frac{\varphi(\delta_k^{-1}|u_k(x)-u_k(y)|)}{|x-y|^{n+1}} \,\dx\,\dy \\
&\leq a\frac{\mathrm{Lip}(u_k)^2}{\delta_k}\int_{\Omega}\int_{\Omega} \frac{1}{\abs{x-y}^{n-1}}\,\dx\,\dy \leq  a \delta_k^{-1/2} \int_{\Omega}\int_{\Omega} \frac{1}{\abs{x-y}^{n-1}}\,\dx\,\dy.
\end{align*}
Hence, $\Rcal_{\delta_k}(u_k) \to 0$ as $k \to \infty$, which concludes the proof.
\end{proof}

\subsection{Aubert \& Kornprobst setting}
\label{sub:ak}
Let $\Omega \subset \R^n$ be a bounded Lipschitz domain. We fix a nonnegative function $\rho:[0,\infty) \to [0,\infty)$  satisfying\vspace{0.2cm}
\begin{itemize}
\item[(H$6_\delta$)] $\rho$ is non-increasing and $\displaystyle \int_{\R^n}\rho(\abs{x})\,\dd x=1$,
\end{itemize}
\vspace{0.2cm}
and consider the regularizers given for $\delta \in \Lambda=(0,\infty)$ and $u \in L^2(\Omega)$  by
\begin{align}\label{Rcal_delta}
\Rcal_{\delta}(u) = \frac{1}{\delta^n}\int_{\Omega} \int_{\Omega} \frac{\abs{u(x)-u(y)}}{\abs{x-y}}\rho\left(\frac{\abs{x-y}}{\delta}\right)\dd x\, \dd y.
\end{align}

\color{black}
\begin{remark}\label{rem:nonlocal}
a) As $\rho$ is non-increasing, we have for all $0<\delta<\bar{\delta}$ and $x,y \in \Omega$ that
$\rho(\abs{x-y}/\delta) \leq \rho(\abs{x-y}/\bar{\delta})$;  consequently,
\[
\Rcal_{\delta}(u) \leq \frac{\bar{\delta}^n}{ \delta^n}\Rcal_{\bar{\delta}}(u)
\]
for all $u \in L^2(\Omega)$. \smallskip

b) Note that the assumption \eqref{eq:assumptions} from Section~\ref{sec:general} is satisfied here; in particular, $\Rcal_{\delta}$ is $L^2$-weakly lower semicontinuous. Indeed, as the dependence of the integrand on $u$ is convex, it is enough to prove strong lower semicontinuity in $L^2(\Omega)$. This is in turn a simple consequence of Fatou's lemma. \smallskip

c) In this set-up, the sets $K_{\delta,j}$  in \eqref{eq:kdeltaj} consist of a single element $w^{(\delta)}_j \in L^2(\Omega)$ in light of the strict convexity of the fidelity term and convexity of $\Rcal_{\delta}$. The upper-level functional from \eqref{Ical_sec5} then becomes
\[
\Ical:(0,\infty) \to [0,\infty), \quad \Ical(\delta) = \norm{w^{(\delta)}-u^c}_{L^2(\Omega;\R^N)}^2.
\]
\end{remark}

The nonlocal functionals in~\eqref{Rcal_delta} have been applied to problems in imaging in~\cite{AuKo09},  providing a derivative-free alternative to popular local models. The localization behavior of these functionals as $\delta \to 0$ is well-studied, originally by Bourgain, Brezis, \& Mironescu \cite{BBM01} and later extended to the $BV$-case in \cite{Dav02, Pon04}. Using these results, we show that, as $\delta \to 0$, the reconstruction functional in our bi-level scheme turns into the $TV$-reconstruction functional, see Proposition~\ref{th:gammanonlocal} below. Moreover, in order to get structural stability inside the domain $\Lambda$, we exploit the monotonicity properties of the functional $\Rcal_{\delta}$, cf.~Remark~\ref{rem:nonlocal}\,a). Lastly, as $\delta \to \infty$, we observe that the regularization term vanishes.

\color{black}
\begin{proposition}[Mosco-convergence of the regularizers]
\label{th:gammanonlocal}
Let $\Lambda =(0,\infty)$, $\Omega \subset \R^n$ be a bounded Lipschitz domain and assume that {\rm (H$6_\delta$)} holds. Then, 
\begin{equation}\label{eq:BNgammanonlocal}
\overline{\Rcal}_{\delta}:=\text{\rm Mosc}(L^2)\text{-}\lim_{\delta'\to \delta}\Rcal_{\delta'}=\begin{cases}
\Rcal_\delta &\text{if} \ \delta \in (0,\infty),\\
\Rcal_{0} &\text{if} \ \delta=0,\\
0 &\text{if $\delta=\infty$},
\end{cases} \quad \text{for $\delta \in \overline{\Lambda}=[0,\infty]$,}
\end{equation}
where 
\begin{equation}\label{Rcal0AK}
\Rcal_0:L^2(\Omega) \to [0,\infty], \quad \Rcal_0(u)=\begin{cases}
\kappa_n \abs{Du}(\Omega), \qquad&\text{if $u \in BV(\Omega)$},\\
\infty \qquad&\text{if $u \in L^2(\Omega) \setminus BV(\Omega),$}
\end{cases}
\end{equation}
with $\kappa_n =\displaystyle \dashint_{\Sbb^{n-1}}| e\cdot \sigma|\, \dd\sigma$ for any $e\in \Sbb^{n-1}$.
\end{proposition}

\begin{proof}
Given $(\delta_k)_k\subset (0,\infty)$ with limit $\delta \in [0,\infty]$, the arguments below, subdivided into three different regimes, show that the Mosco-limit of $(\Rcal_{{\delta}_k})_k$ exists and is equal to the right-hand side of \eqref{eq:BNgammanonlocal}.  \smallskip

\textit{Case 1: $\delta=0$.} For the lower bound, take a sequence $u_k \weakly u$ in $L^2(\Omega)$ and assume without loss of generality that
\[
\sup_k \Rcal_{\delta_k}(u_k) <\infty.
\]
By \cite[Theorem~4]{BBM01}, $(u_k)_k$ is relatively compact in $L^1(\Omega)$, so that $u_k\to u$ in $L^1(\Omega)$. We now use the $\Gamma$-liminf result with respect to the $L^1(\Omega)$-convergence in~\cite[Corollary~8]{Pon04},
to deduce that
\[
\Rcal_0(u) \leq \liminf_{k \to \infty} \Rcal_{\delta_k}(u_k),
\]
as desired. For the recovery sequence, we may suppose that $u \in L^2(\Omega) \cap BV(\Omega)$. Then, it follows from~\cite[Corollary~1]{Pon04} that
\[
\lim_{k \to \infty} \frac{1}{\delta_k^n}\int_{\Omega} \int_{\Omega} \frac{\abs{u(x)-u(y)}}{\abs{x-y}}\rho\left(\frac{\abs{x-y}}{\delta_k}\right)\dd x\, \dd y =  \kappa_n  \abs{Du}(\Omega),
\]
showing that the constant sequence $u_k=u$ for all $k \in \N$ provides a recovery sequence.\smallskip

\textit{Case 2: $\delta \in (0,\infty)$.} For the liminf inequality, take a sequence $(u_k)_k$ converging weakly to $u$ in $L^2(\Omega)$. If $\bar{\delta} \in (0,\delta)$, then $\delta_k >\bar{\delta}$ for all $k \in \N$ large enough. Hence, it follows from Remark~\ref{rem:nonlocal}\,a) that
\[
\liminf_{k \to \infty} \Rcal_{\delta_k}(u_k) \geq \liminf_{k \to \infty} \frac{\bar{\delta}^n}{\delta_k^n}\Rcal_{\bar{\delta}}(u_k) \geq \frac{\bar{\delta}^n}{\delta^n}\Rcal_{\bar{\delta}}(u),
\]
where the last inequality uses the weak lower semicontinuity of $\Rcal_{\bar{\delta}}$, cf.~Remark~\ref{rem:nonlocal}\,b).  Letting $\bar{\delta} \nearrow \delta$ and using the monotone convergence theorem gives
\[
\liminf_{k \to \infty} \Rcal_{\delta_k}(u_k) \geq \Rcal_{\delta}(u).
\]

For the limsup inequality, consider $u \in L^2(\Omega)$ with $\Rcal_{\delta}(u) <\infty$. Since $\rho$ is non-increasing by (H$6_\delta$), we may extend $u$ to a function $\bar{u} \in L^2(\R^n)$ by reflection across the boundary of the Lipschitz domain $\Omega$ such that
\[
\int_{\R^n}\int_{\R^n} \frac{\abs{\bar{u}(x)-\bar{u}(y)}}{\abs{x-y}}\rho\left(\frac{\abs{x-y}}{\delta}\right)\dd x\, \dd y < \infty,
\]
cf.~\cite[Proof of Theorem~4]{BBM01}. With $(\varphi_{\varepsilon})_{\varepsilon}$ a family of smooth standard mollifiers, the sequence $u_l :=\varphi_{1/l}*\bar{u}$ for $l \in \N$ converges to $u$ in $L^2(\Omega)$ as $l \to \infty$, and we may argue similarly to the proof of the smooth approximation property $(iv)$ in Section~\ref{sec:4} to conclude that
\[
\lim_{l \to \infty} \Rcal_{\delta}(u_l) = \Rcal_\delta(u).
\]
\color{black}
With $\rho_{\delta}:=\delta^{-n}\rho(\abs{\cdot}/\delta)$ and for a fixed $l \in \N$, we find
that\begin{align*}
\abs{\Rcal_{\delta}(u_l)-\Rcal_{\delta_k}(u_l)} &\leq \int_{\Omega} \int_{\Omega} \frac{\abs{u_l(x)-u_l(y)}}{\abs{x-y}}\abs{\rho_{\delta}(x-y)-\rho_{\delta_k}(x-y)}\,\dd x\, \dd y \\[0.1cm]
&\leq \mathrm{Lip}(u_l)\abs{\Omega}\norm{\rho_{\delta}-\rho_{\delta_k}}_{L^1(\R^n)},
\end{align*}
where $\mathrm{Lip}(u_l)$ is the Lipschitz constant of $u_l$.
We have $\rho_{\delta_k} \to \rho_{\delta}$ in $L^1(\R^n)$ as $k \to \infty$ by a standard argument approximating $\rho$ with smooth functions. Hence, we obtain
\[
\lim_{k \to \infty} \Rcal_{\delta_k}(u_l)=\Rcal_{\delta}(u_l),
\]
and, letting $l \to \infty$, results in
\[
\lim_{l \to \infty} \lim_{k \to \infty} \Rcal_{\delta_k}(u_l)=\Rcal_{\delta}(u). 
\]
The limsup inequality now follows by extracting an appropriate diagonal sequence. \smallskip

\color{black}
\textit{Case 3: $\delta = \infty$.} The only nontrivial case is the limsup inequality, for which we take a sequence $(u_l)_l \subset C_c^{\infty}(\R^n)$ that converges to $u$ in $L^2(\Omega)$. Then, with $R$ larger than the diameter of $\Omega$, one obtains for every $l \in \N$ that
\begin{align*}
\Rcal_{\delta_k}(u_l)&=\frac{1}{\delta_k^n}\int_{\Omega}\int_{\Omega}\frac{\abs{u_l(x)-u_l(y)}}{\abs{x-y}}\rho\left(\frac{\abs{x-y}}{\delta_k}\right)\dd x\, \dd y\\
&\leq \mathrm{Lip}(u_l)\int_{\Omega}\int_{\Omega/\delta_k}\rho\left(\absB{z-\frac{y}{\delta_k}}\right)\dd z\, \dd y\leq \mathrm{Lip}(u_l)\int_{\Omega}\int_{B_{\frac{R}{\delta_k}}(0)}\rho(\abs{w})\,\dd w\, \dd y.
\end{align*}
As $k \to \infty$, the last quantity goes to zero since $\rho(\abs{\cdot}) \in L^1(\R^n)$. Therefore, we deduce that
\[
\lim_{k \to \infty} \Rcal_{\delta_k}(u_l)=0,
\]
and conclude again with a diagonal argument.
\end{proof}

\subsection{Conclusions and examples}
In both the Brezis \& Nguyen  and  the Aubert \& Kornprobst settings, we now find that the extension $\overline{\Ical}:[0,\infty] \to [0,\infty]$ is given by
\[
\overline{\Ical}(\delta) = \begin{cases}
\Ical(\delta) \qquad&\text{if $\delta \in (0,\infty)$},\\
\norm{w^{(0)}-u^c}_{L^2(\Omega;\R^N)}^2 \qquad&\text{if $\delta = 0$},\\
\norm{u^{\eta}-u^c}^2_{L^2(\Omega;\R^N)} \qquad&\text{if $\delta = \infty$},
\end{cases}
\]
where $w^{(0)}_j$ for $j \in \{1,\dots,N\}$ is the unique minimizer of the $TV$-reconstruction functional $\overline{\Jcal}_{0,j}$ (with different weight factors in the two cases).
In particular, we deduce from Theorem~\ref{theo:relax2} and Corollary~\ref{cor:relax} that $\overline{\Ical}$ is the relaxation of $\Ical$ and that these extended upper-level functionals $\overline{\Ical}$ admit minimizers $\bar{\delta} \in [0,\infty]$. To get an intuition about when this optimal parameter is attained at the boundary or in the interior of $\Lambda$, we present the following examples. 
\begin{example}\label{ex:nonlocal}
a) For both settings analyzed in this section, it is clear that if the noisy and clean image coincide, $u^{c} \equiv u^{\eta}$, then the reconstruction model with parameter $\delta = \infty$ gives the exact clean image back. Hence, in this case the optimal parameter is attained at the boundary point $\delta = \infty$. \smallskip

b) Next, we illustrate the case when the optimal parameter is attained at the boundary point $\delta = 0$. Consider the Aubert \& Kornprobst setting in Subsection \ref{sub:ak} and let $\Omega = (-1,1)$, $N=1$, $u^{c}=0$, and $u^{\eta}(x) = \kappa_nx$ for $x\in (-1,1)$.
The reconstruction of $u^{\eta}$ with the total variation regularizer $\Rcal_0$ in~\eqref{Rcal0AK} is of the form 
\begin{align*}
w^{(0)}= \max\{\theta_1, \min\{\theta_2, u^\eta\}\} \quad \text{ for some $\theta_1,\theta_2 \in \R$.}
\end{align*} 
To see this, we observe that $\overline{\Jcal}_0(\tilde u)\leq \overline{\Jcal}_{0}(u)$  for any $u \in BV(-1,1)$ with 
\[
\tilde{u} = \max\{u^{-}, \min\{u^{+}, u^\eta\}\},
\]
where $u^{-} := \mathrm{ess\,inf}_{x \in (-1,1)}u(x)$ and $u^{+} := \esssup_{x \in (-1,1)}u(x)$. Indeed, the map $\tilde{u}$ has at most the same total variation as $u$ and does not increase the distance to $u^{\eta}$ anywhere. Next, since $u^\eta$ is an odd function, the same should hold for the minimizer, meaning that $-\theta_1=\theta_2=:\theta \in [0,\kappa_n]$.
We can now determine the value of $\theta$ by optimizing the quantity $\overline{\Jcal}_{0}(w^{(0)})$ in $\theta$. This boils down to minimizing
\[
\frac{2}{3}\kappa_n^2\left(1-\frac{\theta}{\kappa_n}\right)^3 + 2\kappa_n\theta,
\]
and yields $\theta = 0$.  Hence, the reconstruction model for $\delta = 0$ yields the exact clean image, so that $\overline{\Ical}(0)=0$.
The same conclusions can be drawn for the Brezis \& Nguyen setting by replacing $\kappa_n$ in the example above with $K(\varphi)$. \smallskip

c) Let us finally address the case when $\overline{\Ical}$ becomes minimal inside $\Lambda=(0,\infty)$.  We work once again with the Aubert \& Kornprobst model from Subsection \ref{sub:ak}, and assume in addition to (H$6_\delta$) that the function $\rho$ is equal to $1$ in a neighborhood of zero. We consider the following conditions on the  pair of data points $(u^c,u^{\eta}) \in L^2(\Omega;\R^N)\times L^2(\Omega;\R^N)$:\smallskip
\begin{itemize}
\item[(H$7_\delta$)] $\norm{u^\eta-u^{c}}_{L^2(\Omega;\R^N)}^2 < \norm{w^{(0)}-u^c}_{L^2(\Omega;\R^N)}^2;$  \smallskip \item[(H$8_\delta$)] $\sum_{j=1}^N\widetilde{\Rcal}(u^c_j)<\sum_{j=1}^{N}\widetilde{\Rcal}(u^\eta_j);$
\end{itemize}
\smallskip
here, $w^{(0)}$ is the componentwise minimizer of the $TV$-reconstruction functional $\overline{\Jcal}_0$ and we set
\begin{align}\label{Rcaltilde_delta}
\widetilde{\Rcal}(u):=\int_{\Omega}\int_{\Omega} \frac{\abs{u(x)-u(y)}}{\abs{x-y}} \,\dd x \,\dd y \quad\text{for $u\in L^2(\Omega)$.}
\end{align}
The two hypotheses above can be realized, for example, by taking $u^{\eta} =(1+\varepsilon)u^{c}$ for some small $\varepsilon >0$ and $w^{(0)}\not = u^c$. 

Notice that (H$7_\delta$) immediately rules out $\delta=0$ as an optimal candidate, since the reconstruction at $\delta=\infty$ is better.  On the other hand, $\rho$ is supposed to be equal to $1$ near the zero, so that we  infer for large enough $\delta$ that
\begin{align}\label{Rcal_delta2}
\Rcal_{\delta}(u)=\frac{1}{\delta^n}\int_{\Omega}\int_{\Omega}\frac{\abs{u(x)-u(y)}}{\abs{x-y}}\,\dd x\, \dd y = \frac{1}{\delta^n}\widetilde{\Rcal}(u)
\end{align}
for all $u \in L^2(\Omega)$. Since, for large $\delta$, the dependence of the regularizer on $\delta$ is of the same type as the weight case from Section~\ref{sec:weight}, we may apply Lemma~\ref{lemma:not-min}\,$(i)$ in view of~(H$8_\delta$). This yields, for all $\delta$ large enough, that
\[
\norm{u^c-w^{(\delta)}}^2_{L^2(\Omega;\R^N)} < \norm{u^c-u^\eta}^2_{L^2(\Omega;\R^N)},
\]
with $w^{(\delta)}$ the minimizer of $\Jcal_\delta$.  This shows that the optimal parameter is not attained at $\delta=\infty$ either and, as a result, needs to be attained inside $\Lambda=(0,\infty)$. Hence, the optimal regularizer lies within the class we started with.

The same conclusions can be drawn for the Brezis \& Nguyen case described in Subsection \ref{sub:BN} if we assume that $\varphi(t) = ct^r$ for small $t$ with $c>0$ and $r \geq 2$. One may take, for instance, the normalized version of the second function in Example~\ref{egAcal}. 
We then suppose that the pair of data points $(u^c, u^\eta)$ satisfies (H$7_\delta$)  and (H$8_\delta$), but now instead of~\eqref{Rcaltilde_delta}, take
\begin{align}\label{Rcaltilde_delta}
\widetilde{\Rcal}(u):=c\int_{\Omega}\int_{\Omega} \frac{\abs{u(x)-u(y)}^r}{\abs{x-y}^{n+1}}\, \dd x\, \dd y\quad\text{for $u\in L^2(\Omega)$.}
\end{align}
 
We observe with $l=\norm{u^{\eta}}_{L^{\infty}(\Omega;\R^N)}$ (which we assume to be finite) and $T^l$ the truncation as in the proof of Proposition~\ref{th:BNgammanonlocal} that
\[
\Jcal_{\delta}(T^l \circ u) \leq \Jcal_{\delta}(u)
\]
for all $u \in L^2(\Omega)$ and $\delta \in (0,\infty)$. Therefore, we may restrict our analysis to functions $u \in L^2(\Omega)$ with $\abs{u(x)-u(y)} \leq 2l$ for all $x,y \in \Omega$. By additionally considering $\delta$ large enough, we now find
\[
\varphi\left(\frac{\abs{u(x)-u(y)}}{\delta}\right)=c\,\frac{\abs{u(x)-u(y)}^r}{\delta^r};
\]
 hence,
\[
\Rcal_{\delta}(u) = \frac{c}{\delta^{r-1}}\int_{\Omega}\int_{\Omega} \frac{\abs{u(x)-u(y)}^r}{\abs{x-y}^{n+1}} \,\dd x \, \dd y =\frac{1}{\delta^{r-1}}\widetilde{\Rcal}(u)
\] 
in analogy to~\eqref{Rcal_delta2}.
\end{example}

\section{Tuning the fractional parameter}
\label{sec: fractional}

This final section revolves around regularization via the $L^2$-norm of the spectral fractional Laplacian of order $s/2$, with $s$ in the parameter range $\Lambda =(0,1)$. Our aim here is twofold. First, we determine the Mosco-limits of the regularizers, which allows us to conclude in view of the general theory in Section~\ref{sec:general} that the extended bi-level problem recovers local models at the boundary points of $\overline{\Lambda}=[0,1]$. Second, we provide analytic conditions ensuring that the optimal parameter lies in the interior of $(0,1)$, and illustrate them with an explicit example.

The motivation behind the fractional Laplacian as a regularizer comes from \cite{AnB17}, where the authors show that replacing the total variation in the classical ROF model~\cite{ROF92} with a spectral fractional Laplacian can lead to comparable reconstruction results with a much smaller computational cost, if the order is chosen correctly. An abstract optimization of the fractional parameter for  the spectral fractional Laplacian has already been undertaken in \cite{BaW20}, although we remark that a convex penalization term is added there to the model  to ensure that the optimal fractional parameter lies inside $(0,1)$.

We begin with the problem set-up. Let $\Omega \subset \R^n$ be a bounded Lipschitz domain and let $(\psi_m)_{m \in \N} \subset H_0^1(\Omega)$ be a sequence of eigenfunctions associated with the Laplace operator  $(-\Delta)$ forming an orthonormal basis of $L^2(\Omega)$. With the corresponding eigenvalues  $0<\lambda_1 \leq \lambda_2\leq \lambda_3\leq  \dots \nearrow \infty$, it holds  for every $m\in \N$ that
\begin{equation}\label{eq:eigenbasis}
\begin{cases}
(-\Delta)\psi_m = \lambda_m \psi_m \qquad&\text{in $\Omega$},\\
\psi_m = 0 \qquad&\text{on  $\partial \Omega$}.
\end{cases}
\end{equation}
Denoting the projection of any $u\in L^2(\Omega)$ onto the $m$th eigenfunction $\psi_m$ by
\begin{align*}
\hat u_m:=\langle u,\psi_m \rangle_{L^2(\Omega)},
\end{align*}
we have the representation $u=\sum_{m=1}^\infty \hat u_m \psi_m$.

With this at hand, one can define for $s  \in (0,1)$ the fractional Sobolev spaces
\begin{align*}
\Hbb^{ s }(\Omega):=\Bigl\{u = \sum_{m=1}^{\infty}\hat{u}_m\psi_m \in L^2(\Omega) \, : \, \sum_{m=1}^{\infty}\lambda_m^{ s }\hat{u}_m^2<\infty\Bigr\},
\end{align*} 
endowed with the inner product
\[
\dprlr{u,v}_{\Hbb^{ s }(\Omega)} := \sum_{m=1}^{\infty}\lambda_m^{ s }\hat{u}_m \hat{v}_m.
\]
It holds that $\Hbb^s(\Omega)$ is a Hilbert space for every $s\in (0,1)$; for more details on these spaces, we refer, e.g.,~to~\cite{CaSt16, NOS15}.
In view of \eqref{eq:eigenbasis}, the so-called spectral fractional Laplacian of order $ s /2$ (with Dirichlet boundary conditions) on these spaces is defined as
\[
(-\Delta_D)^{ s /2}:\Hbb^{ s }(\Omega) \to L^2(\Omega), \quad (-\Delta_D)^{ s /2}u = \sum_{m=1}^{\infty}\lambda_{m}^{ s /2}\hat{u}_m\psi_m.
\]

For $ s  \in (0,1)$, we consider the regularizer 
\begin{align}\label{Ralpha6}
\Rcal_{ s }: L^2(\Omega) \to [0,\infty], \quad \Rcal_{ s }(u) = \begin{cases}
\mu\norm{(-\Delta_{D})^{ s /2}u}_{L^2(\Omega)}^2 &\text{for $u \in \Hbb^{ s }(\Omega)$},\\
\infty &\text{otherwise},
\end{cases}
\end{align}
with some $\mu >0$. At the end of this section (see Remark~\ref{rmk:str-pres}), the weight parameter \(\mu\) will be used to exhibit examples where structure preservation holds.
The regularizers $\Rcal_{s}$ coincide with $\mu \norm{\cdot}^2_{\Hbb^{ s }(\Omega)}$ on $\Hbb^{ s }(\Omega)$, and are $L^2$-weakly lower semicontinuous because $u_k \weakly u$ in $L^2(\Omega)$ yields\[
\liminf_{k \to \infty} \Rcal_s(u_k)= \liminf_{k \to \infty}\mu\sum_{m=1}^{\infty}\lambda^s_m\widehat{(u_k)}_m^2\geq \mu \sum_{m=1}^{\infty}\lambda^s_m\widehat{u}_m^2=\Rcal_s(u)
\]
by a discrete version of Fatou's lemma. Therefore, the hypotheses in \eqref{eq:assumptions} from Section~\ref{sec:general} are satisfied.

Next, we determine the Mosco-limits of the regularizers, and thereby, provide the basis for extending the upper-level functional according to Section~\ref{sec:general}. 

\begin{proposition}[\boldmath{Mosco}-convergence of the regularizers]\label{prop:gammafractional}
Let \( \Lambda:=(0,1)\) and $\Rcal_s$ for each \( s  \in \Lambda\) be given by \eqref{Ralpha6}. Then, for $u\in L^2(\Omega)$ and $s\in \overline{\Lambda}=[0,1]$,
\begin{equation}\label{eq:gammafractional}
\overline{\Rcal}_s(u)={\rm Mosc}(L^2)\text{-}\lim_{ s ' \to  s } \Rcal_{s'}(u)=\begin{cases}
\Rcal_{ s}(u) &\text{if} \  s  \in (0,1),\\
\mu\norm{u}^2_{L^2(\Omega)}&\text{if} \  s  =0,\\
\mu\norm{\nabla u}^2_{L^2(\Omega)} + \chi_{H^1_0(\Omega)}(u) &\text{if} \  s =1.
\end{cases}
\end{equation}
\end{proposition}

\begin{proof}
Let us observe up front that for all $u \in L^2(\Omega)$,
\begin{equation}\label{eq:regident}
\norm{u}^2_{L^2(\Omega)}=\sum_{m=1}^{\infty} \widehat{u}_m^2
\quad \text{and} \quad \norm{\nabla u}^2_{L^2(\Omega)} + \chi_{H^1_0(\Omega)}(u)=\sum_{m=1}^{\infty} \lambda_m\widehat{u}_m^2;
\end{equation}
indeed, the  first formula is simply Parseval's identity, while the second one is a consequence of  $\nabla u=\sum_{m=1}^\infty \hat u_m\nabla \psi_m$ for $u\in H_0^1(\Omega)$ and the orthogonality  in $L^2(\Omega;\R^n)$ of the gradients
$(\nabla \psi_m)_m$ with
\[
\norm{\nabla \psi_m}^2_{L^2(\Omega;\R^n)} =-\int_{\Omega} \psi_m\, \Delta\psi_m\,\dd{x}=\int_\Omega\lambda_m\psi_m^2\,\dd{x} =\lambda_m.
\]

Fixing a sequence $( s _k)_k\subset (0,1)$ with limit $ s  \in [0,1]$, we want to prove now that the  Mosco-limit of $(\Rcal_{{ s_k }})_k$ exists and is given by the right-hand side of \eqref{eq:gammafractional}. 
\smallskip

\textit{Step 1: The liminf-inequality.} Let $u_k \rightharpoonup u$ in $L^2(\Omega)$, and assume without loss of generality that $\liminf_{k\to \infty}\Rcal_{ s _k}(u_k)<\infty$. Then, since $\widehat{(u_k)}_m \to \widehat{u}_m$ for each $m \in \N$ as $k \to \infty$, it follows from a discrete version of Fatou's lemma that
\[
\infty > \liminf_{k\to \infty}\Rcal_{s_k}(u_k) = \liminf_{k \to \infty} \mu\sum_{m=1}^{\infty} \lambda_m^{s_k}\widehat{(u_k)}_m^2 \geq \mu\sum_{m=1}^{\infty} \lambda_m^s \widehat{u}_m^2.
\]
In light of \eqref{eq:regident} for the cases $s\in\{0,1\}$, the last quantity equals the regularizer on the right hand side of \eqref{eq:gammafractional} in all the three regimes. This finishes the proof of the lower bound. \smallskip

\textit{Step 2: Construction of a recovery sequence}. We first consider the $u\in H^1_0(\Omega)$ case. By the regularity of \(u\) and Lebesgue's dominated converge theorem (applied
to the counting measure), we get
\[
\lim_{k \to \infty}\Rcal_{ s _k}(u)=\lim_{k \to \infty}\mu\sum_{m=1}^{\infty}
\lambda_m^{ s _k}\widehat{u}_m^2=\mu\sum_{m=1}^{\infty}
\lambda_m^{ s }\widehat{u}_m^2,
\]
which concludes the proof for $u\in H^1_0(\Omega)$.

In the general case where $u\in \Hbb^{ s }(\Omega)$, we consider the sequence $(u_l)_l \subset H^1_0(\Omega)$ defined by 
$u_l:=\sum_{m=1}^l \hat{u}_m\psi_m$ for every $l\in \mathbb{N}$. Then, by construction, $u_l\to u$ strongly in $L^2(\Omega)$ and  
\[
\lim_{l \to \infty} \sum_{m=1}^{\infty} \lambda_m^{s}\widehat{(u_l)}_m^2=\lim_{l \to \infty}\sum_{m=1}^l \lambda_m^s \widehat{u}_m^2=\sum_{m=1}^{\infty} \lambda_m^{s} \widehat{u}_m^2.
\] 
The existence of a recovery sequence follows then by classical diagonalization arguments, using the previous case. 
\end{proof}

Given clean and  noisy images, $\ucl$, $\uet \in L^2(\Omega;\R^N)$, we work with the reconstruction functionals
\[\Jcal_{ s , j}:L^2(\R^n) \to [0,\infty],\quad 
\Jcal_{ s ,j}(u)=\norm{u-\uet_j}_{L^2(\Omega)}^2+\Rcal_{ s }(u)
\]
for $s\in (0,1)$ and $j\in\{1, \ldots, N\}$.  
 Recalling \eqref{training} and~\eqref{trainingextended}, we obtain as a consequence of Proposition~\ref{prop:gammafractional}  that the extension of the upper-level functional $\Ical$ to $\overline{\Lambda}$ is given by
\begin{equation*}
\overline{\Ical}:[0,1] \to [0,\infty], \qquad \overline{\Ical}( s ) = \norm{ w^{(s)} - \ucl}_{L^2(\Omega;\R^N)}^2;
\end{equation*}
here, $w^{(s)}=(w^{(s)}_1, \ldots, w^{(s)}_N)$ with $w^{(s)}_j$ the unique minimizer of the strictly convex functional
\begin{align}\label{reconstructJsj}
\overline{\Jcal}_{ s ,j}(u)=\norm{u-\uet_j}_{L^2(\Omega)}^2+\overline{\Rcal}_{ s }(u) =\sum_{m=1}^{\infty} (\hat{u}_m-\widehat{(\uet_j)}_m)^2+\mu\lambda_m^{ s }\hat{u}_m^2.
\end{align}  
By Theorem~\ref{theo:relax2}, $\overline{\Ical}$ is then the relaxation of $\Ical$ and has a minimizer in $\overline{\Lambda}=[0,1]$.

We now continue by exhibiting conditions under which the minimum of $\overline{\Ical}$ is attained inside $(0,1)$. 
This is based on a direct approach, observing that the components of $w^{(s)}$ can be determined explicitly by minimizing the entries of the sum in \eqref{reconstructJsj} individually. This gives the representation
\begin{equation}\label{eq:fracminimizer}
 w^{(s)}_j = \sum_{m=1}^{\infty} \frac{1}{1+\mu\lambda_m^{ s }}\widehat{(\uet_j)}_m\psi_m\qquad \text{for $j\in\{1, \ldots, N\}$. }
\end{equation}
The following lemma investigates how $ w^{(s)}$ varies with $ s $. In the  $s>0$ case, this lemma  is essentially  contained in \cite[Theorem~2]{BaW20} (i.e., in a slightly different setting
with periodic instead of Dirichlet boundary conditions). The proof below contains some additional details for the reader's convenience. 
\begin{lemma}
\label{lemma:der}
Assume that $u^\eta\in  \Hbb^{\varepsilon}(\Omega;\R^N)$ for some $\varepsilon \in (0,1)$. Then, the map $[0,1]\mapsto L^2(\Omega;\R^N),\ s\mapsto w^{(s)}$ is Fr\'echet-differentiable with derivative
\begin{equation}\label{eq:derivformula}
\partial_s  w^{(s)}=-\sum_{m=1}^{\infty} \frac{\mu\log(\lambda_m)\lambda_m^{ s }}{(1+\mu\lambda_m^{ s })^2} \widehat{\uet}_m\psi_m. 
\end{equation} 
\end{lemma}
\begin{proof}
For $j\in \{1, \ldots, N\}$, we set $$v_j:=-\sum_{m=1}^{\infty} \frac{\mu\log(\lambda_m)\lambda_m^{ s }}{(1+\mu\lambda_m^{ s })^2}\widehat{(\uet_j)}_m\psi_m,$$ which is a well-defined element of $L^2(\Omega)$ for all $s \in [0,1]$ because $u^{\eta}_j \in \Hbb^{\varepsilon}(\Omega)$. Since
\[
\frac{\ubet_j -  w^{(s)}_j}{t -  s } = \sum_{m=1}^{\infty}\frac{1}{t- s }\left(\frac{1}{1+\mu\lambda^{t}_m}-\frac{1}{1+\mu\lambda^{ s }_m}\right)\widehat{(\uet_j)}_m\psi_m, \quad s, t\in[0,1],
\]
in view of~\eqref{eq:fracminimizer}, we can apply the mean value theorem to obtain, for each $m\in \mathbb{N}$, a value $\gamma$ in between $ s $ and $t$ such that
\[
\absBB{\frac{1}{t- s }\left(\frac{1}{1+\mu\lambda^{t}_m}-\frac{1}{1+\mu\lambda^{ s }_m}\right)} \leq \absBB{\frac{\mu\log(\lambda_m)\lambda_m^{\gamma}}{(1+\mu\lambda_m^{\gamma})^2}}\leq  \abs{\log(\lambda_m)}.
\]
Exploiting once again that $u^\eta_j\in \Hbb^{\varepsilon}(\Omega)$ gives\begin{align*}
\normBB{\frac{\ubet_j -  w^{(s)}_j}{t -  s }-v_j}_{L^{2}(\Omega)}^2&=\sum_{m=1}^{\infty}\absBB{\frac{1}{t- s }\left(\frac{1}{1+\mu\lambda^{t}_m}-\frac{1}{1+\mu\lambda^{ s }_m}\right)+\frac{\mu\log(\lambda_m)\lambda_m^{ s }}{(1+\mu\lambda_m^{ s })^2}}^2\widehat{(\uet_j)}_m^2\\
&\leq\sum_{m=1}^{\infty}\abs{\log(\lambda_m)}^2\widehat{(\uet_j)}_m^2 <\infty.
\end{align*}
In particular, we may take the limit $t\to s$ on the left-hand side of the preceding estimate and interchange with the sum to show the claim.
\end{proof}
It follows as a consequence of Lemma \ref{lemma:der} that the upper level function $\overline{\Ical}:[0,1]\to [0,\infty]$ is differentiable with derivative
\[
\overline{\Ical}'(s)
= 2\left\langle\partial_s  w^{(s)},  w^{(s)}-\ucl\right\rangle_{L^2(\Omega;\R^N)}
\]
for $s\in [0,1]$; at the boundary points $s=0$ and $s=1$, $\overline{\Ical}'(s)$ stands for the one-sided derivative.
Therefore, the simple conditions
\[
 \overline{\Ical}'(0) < 0 \quad \text{and} \quad  
  \overline{\Ical}'(1)  > 0,
\]
imply that $\overline{\Ical}$ does not attain its minimizer at $ s  =0$ or at $ s  = 1$, respectively. Due to~\eqref{eq:derivformula} and~\eqref{eq:fracminimizer}, these requirements can be written as follows:
\begin{itemize}
\item[(H$1_s$)] $\displaystyle\sum_{j=1}^N\sum_{m=1}^{\infty}\log(\lambda_m)\widehat{(\uet_j)}_m\left(\widehat{(\uet_j)}_m-(1+\mu)\widehat{(\ucl_j)}_m\right) >0;$ 
\item[(H$2_s$)] $\displaystyle\sum_{j=1}^N\sum_{m=1}^{\infty}\frac{\log(\lambda_m)\lambda_m}{(1+\mu\lambda_m)^3}\widehat{(\uet_j)}_m\left(\widehat{(\uet_j)}_m-(1+\mu\lambda_m)\widehat{(\ucl_j)}_m\right) <0.$
\end{itemize}
\color{black}
Since (H$1_s$) guarantees that the minimizer of $\overline{\Ical}$ is not $ s =0$ and (H$2_s$) ensures the minimizer to be different from $ s  =1$, Corollary~\ref{cor:relax}\,$(iii)$ yields the following result.
\begin{corollary}
Suppose that $\uet \in \Hbb^{\varepsilon}(\Omega;\RR^N)$ for some $\varepsilon\in (0,1)$, and that assumptions {\rm (H$1_s$)} and {\rm (H$2_s$)} are satisfied. Then, $\Ical$ admits a minimizer $\bar{ s } \in (0,1)$.
\end{corollary}

We close this section with an interpretation of the conditions (H$1_s$) and (H$2_s$), and a specific example in which they are both satisfied. 

\begin{remark}\label{rmk:str-pres}
a) Suppose that $N=1$. Decomposing the noisy image into the sum of the clean image and the noise, i.e., $\uet = \ucl+\eta$, turns (H$1_s$) and (H$2_s$) into
\begin{equation}\label{eq:noisecond}
\begin{cases}
\displaystyle\sum_{m=1}^{\infty}\log(\lambda_m)\left(-\mu \widehat{\ucl}_m^2 + (1-\mu)\widehat{\ucl}_m\widehat{\eta}_m+\widehat{\eta}_m^2\right)>0,\\
\displaystyle\sum_{m=1}^{\infty}\frac{\log(\lambda_m)\lambda_m}{(1+\mu\lambda_m)^3}\left(-\mu\lambda_m \widehat{\ucl}_m^2 + (1-\mu\lambda_m)\widehat{\ucl}_m\widehat{\eta}_m+\widehat{\eta}_m^2\right)<0.
\end{cases}
\end{equation}
If we assume that the noise has mostly high frequencies and that the clean image has mostly moderate frequencies, then the mixed terms in~\eqref{eq:noisecond} will be small. The first condition is then close to
\[
-\mu\sum_{m=1}^{\infty}\log(\lambda_m)\widehat{\ucl}_m^2 + \sum_{m=1}^{\infty}\log(\lambda_m)\widehat{\eta}_m^2 >0,
\]
which holds for sufficiently small $\mu$. Similarly, for sufficiently large $\mu$, the second condition is satisfied. As we analyse in b) below, there are instances where we can find a range for \(\mu\) that implies both conditions.  \smallskip 

b) In the case where $\Omega = (0,\pi)^2$, by indexing the eigenfunctions via $m = (m_1,m_2) \in \Z^2, \ m_1, m_2 \not = 0$, we find
\[
\psi_{m}(x) = \sin(m_1x_1)\sin(m_2x_2)
\]
with corresponding eigenvalues $\lambda_m = m_1^2 + m_2^2$. By choosing $u^c = \psi_{(1,1)}$  as the clean image and  $\eta = \frac{1}{10}\psi_{(10,10)}$ as the noise, the condition~\eqref{eq:noisecond} turns into
\[
\begin{cases}
-100 \,\mu \log(2)+\log(200) >0,\\
\displaystyle -\mu \frac{4\log(2)}{(1+2\mu)^3}+\frac{2\log(200)}{(1+200\mu)^3}<0,
\end{cases}
\]
which is satisfied for
\[
0.0236 \approx \mu_{-} < \mu < \mu_{+} \approx 0.0764.
\]
On the other hand, when $\mu = 0.023$, then $ s  = 1$ is optimal, while the optimal solution for $\mu
= 0.11$  is $ s  = 0$. This can be seen numerically as for these values of $\mu$, the derivative $\overline{\Ical}'$ is either negative or positive on $[0,1]$, respectively. 
\end{remark}
 
 \section*{Acknowledgements}
The work of E.D. has been partially supported by the Austrian Science Fund (FWF) through the grants F65, V 662, Y1292, and I 4052.
R.F. was partially supported by King Abdullah University of Science and Technology (KAUST) baseline funds and KAUST OSR-CRG2021-4674.


\bibliographystyle{abbrv}



\begin{thebibliography}{90}

\bibitem{AnB17}
H.~Antil and S.~Bartels.
\newblock Spectral approximation of fractional {PDE}s in image processing and
  phase field modeling.
\newblock {\em Comput. Methods Appl. Math.}, 17(4):661--678, 2017.

\bibitem{ADK20}
H.~Antil, Z.~W. Di, and R.~Khatri.
\newblock Bilevel optimization, deep learning and fractional {L}aplacian
  regularization with applications in tomography.
\newblock {\em Inverse Problems}, 36(6):064001, 22, 2020.

\bibitem{Antiletal22}
H.~Antil, H.~D\'{i}az, T.~Jing, and A.~Schikorra.
\newblock Nonlocal bounded variations with applications.
\newblock {\em Preprint, arXiv:2208.11746}, 2022.

\bibitem{AnR19}
H.~Antil and C.~N. Rautenberg.
\newblock Sobolev spaces with non-{M}uckenhoupt weights, fractional elliptic
  operators, and applications.
\newblock {\em SIAM J. Math. Anal.}, 51(3):2479--2503, 2019.

\bibitem{AuKo09}
G.~Aubert and P.~Kornprobst.
\newblock Can the nonlocal characterization of {S}obolev spaces by {B}ourgain
  et al. be useful for solving variational problems?
\newblock {\em SIAM J. Numer. Anal.}, 47(2):844--860, 2009.

\bibitem{BaW20}
S.~Bartels and N.~Weber.
\newblock Parameter learning and fractional differential operators:
  Applications in regularized image denoising and decomposition problems.
\newblock {\em Mathematical Control and Related Fields}, 2021.

\bibitem{BeMC14}
J.~C. Bellido and C.~Mora-Corral.
\newblock Existence for nonlocal variational problems in peridynamics.
\newblock {\em SIAM J. Math. Anal.}, 46(1):890--916, 2014.

\bibitem{BeMC18}
J.~C. Bellido and C.~Mora-Corral.
\newblock Lower semicontinuity and relaxation via {Y}oung measures for nonlocal
  variational problems and applications to peridynamics.
\newblock {\em SIAM J. Math. Anal.}, 50(1):779--809, 2018.

\bibitem{BBF93}
M.~Belloni, G.~Buttazzo, and L.~Freddi.
\newblock Completion by gamma-convergence for optimal control problems.
\newblock {\em Ann. Fac. Sci. Toulouse Math. (6)}, 2(2):149--162, 1993.

\bibitem{BeB18}
M.~Benning and M.~Burger.
\newblock Modern regularization methods for inverse problems.
\newblock {\em Acta Numer.}, 27:1--111, 2018.

\bibitem{BEPS11}
J.~Boulanger, P.~Elbau, C.~Pontow, and O.~Scherzer.
\newblock Non-local functionals for imaging.
\newblock In {\em Fixed-point algorithms for inverse problems in science and
  engineering}, volume~49 of {\em Springer Optim. Appl.}, pages 131--154.
  Springer, New York, 2011.

\bibitem{BBM01}
J.~Bourgain, H.~Brezis, and P.~Mironescu.
\newblock Another look at {S}obolev spaces.
\newblock In {\em Optimal control and partial differential equations}, pages
  439--455. IOS, Amsterdam, 2001.

\bibitem{Bra02}
A.~Braides.
\newblock {\em {$\Gamma$}-convergence for beginners}, volume~22 of {\em Oxford
  Lecture Series in Mathematics and its Applications}.
\newblock Oxford University Press, Oxford, 2002.

\bibitem{BKP10}
K.~Bredies, K.~Kunisch, and T.~Pock.
\newblock Total generalized variation.
\newblock {\em SIAM J. Imaging Sci.}, 3(3):492--526, 2010.

\bibitem{BrNg18}
H.~Brezis and H.-M. Nguyen.
\newblock Non-local functionals related to the total variation and connections
  with image processing.
\newblock {\em Ann. PDE}, 4(1):Art. 9, 77, 2018.

\bibitem{But87}
G.~Buttazzo.
\newblock Some relaxation problems in optimal control theory.
\newblock {\em J. Math. Anal. Appl.}, 125(1):272--287, 1987.

\bibitem{BuD82}
G.~Buttazzo and G.~Dal~Maso.
\newblock {$\Gamma $}-convergence and optimal control problems.
\newblock {\em J. Optim. Theory Appl.}, 38(3):385--407, 1982.

\bibitem{CaSt16}
L.~A. Caffarelli and P.~R. Stinga.
\newblock Fractional elliptic equations, caccioppoli estimates and regularity.
\newblock {\em Annales de l'Institut Henri Poincar\'e C, Analyse non
  lin\'eaire}, 33:767--807, 2016.

\bibitem{ChL97}
A.~Chambolle and P.-L. Lions.
\newblock Image recovery via total variation minimization and related problems.
\newblock {\em Numer. Math.}, 76(2):167--188, 1997.

\bibitem{ChPaPr04}
T.~Champion, L.~De~Pascale, and F.~Prinari.
\newblock {$\Gamma$}-convergence and absolute minimizers for supremal
  functionals.
\newblock {\em ESAIM Control Optim. Calc. Var.}, 10(1):14--27, 2004.

\bibitem{ChPoRaBi13}
Y.~Chen, T.~Pock, R.~Ranftl, and H.~Bischof.
\newblock Revisiting loss-specific training of filter-based mrfs for image
  restoration.
\newblock In J.~Weickert, M.~Hein, and B.~Schiele, editors, {\em Pattern
  Recognition}, pages 271--281, Berlin, Heidelberg, 2013. Springer Berlin
  Heidelberg.

\bibitem{ChRaPo14}
Y.~Chen, R.~Ranftl, and T.~Pock.
\newblock Insights into analysis operator learning: from patch-based sparse
  models to higher order {MRF}s.
\newblock {\em IEEE Trans. Image Process.}, 23(3):1060--1072, 2014.

\bibitem{ChDeSc17}
C.~V. Chung, J.~C. De~los Reyes, and C.~B. Sch\"{o}nlieb.
\newblock Learning optimal spatially-dependent regularization parameters in
  total variation image denoising.
\newblock {\em Inverse Problems}, 33(7):074005, 31, 2017.

\bibitem{Dac08}
B.~Dacorogna.
\newblock {\em Direct methods in the calculus of variations}, volume~78 of {\em
  Applied Mathematical Sciences}.
\newblock Springer, New York, second edition, 2008.

\bibitem{Dal93}
G.~Dal~Maso.
\newblock {\em An introduction to {$\Gamma$}-convergence}, volume~8 of {\em
  Progress in Nonlinear Differential Equations and their Applications}.
\newblock Birkh\"{a}user Boston, Inc., Boston, MA, 1993.

\bibitem{Dav02}
J.~D\'{a}vila.
\newblock On an open question about functions of bounded variation.
\newblock {\em Calc. Var. Partial Differential Equations}, 15(4):519--527,
  2002.

\bibitem{DFL19}
E.~Davoli, I.~Fonseca, and P.~Liu.
\newblock Adaptive image processing: first order {PDE} constraint regularizers
  and a bilevel training scheme.
\newblock {\em Preprint, arXiv:1902.01122}, 2019.

\bibitem{DaL18}
E.~Davoli and P.~Liu.
\newblock One dimensional fractional order {$TGV$}: gamma-convergence and
  bilevel training scheme.
\newblock {\em Commun. Math. Sci.}, 16(1):213--237, 2018.

\bibitem{DeSc13}
J.~C. De~los Reyes and C.-B. Sch\"{o}nlieb.
\newblock Image denoising: learning the noise model via nonsmooth
  {PDE}-constrained optimization.
\newblock {\em Inverse Probl. Imaging}, 7(4):1183--1214, 2013.

\bibitem{DeScVa16}
J.~C. De~Los~Reyes, C.-B. Sch\"{o}nlieb, and T.~Valkonen.
\newblock The structure of optimal parameters for image restoration problems.
\newblock {\em J. Math. Anal. Appl.}, 434(1):464--500, 2016.

\bibitem{DeScVa17}
J.~C. De~los Reyes, C.-B. Sch\"{o}nlieb, and T.~Valkonen.
\newblock Bilevel parameter learning for higher-order total variation
  regularisation models.
\newblock {\em J. Math. Imaging Vision}, 57(1):1--25, 2017.

\bibitem{DDM21}
M.~D'Elia, J.~C. De~Los~Reyes, and A.~Miniguano-Trujillo.
\newblock Bilevel parameter learning for nonlocal image denoising models.
\newblock {\em J. Math. Imaging Vision}, 63(6):753--775, 2021.

\bibitem{DeZ20}
S.~Dempe and A.~Zemkoho, editors.
\newblock {\em Bilevel optimization---advances and next challenges}, volume 161
  of {\em Springer Optimization and Its Applications}.
\newblock Springer, Cham, [2020] \copyright 2020.

\bibitem{DPV12}
E.~Di~Nezza, G.~Palatucci, and E.~Valdinoci.
\newblock Hitchhiker's guide to the fractional {S}obolev spaces.
\newblock {\em Bull. Sci. Math.}, 136(5):521--573, 2012.

\bibitem{Do12}
J.~Domke.
\newblock Generic methods for optimization-based modeling.
\newblock In N.~D. Lawrence and M.~Girolami, editors, {\em Proceedings of the
  Fifteenth International Conference on Artificial Intelligence and
  Statistics}, volume~22 of {\em Proceedings of Machine Learning Research},
  pages 318--326, La Palma, Canary Islands, 21--23 Apr 2012. PMLR.

\bibitem{Elb11}
P.~Elbau.
\newblock Sequential lower semi-continuity of non-local functionals, 2011.

\bibitem{FoL07}
I.~Fonseca and G.~Leoni.
\newblock {\em Modern methods in the calculus of variations: {$L^p$} spaces}.
\newblock Springer Monographs in Mathematics. Springer, New York, 2007.

\bibitem{GiOs08}
G.~Gilboa and S.~Osher.
\newblock Nonlocal operators with applications to image processing.
\newblock {\em Multiscale Model. Simul.}, 7(3):1005--1028, 2008.

\bibitem{HiR17}
M.~Hinterm\"{u}ller and C.~N. Rautenberg.
\newblock Optimal selection of the regularization function in a weighted total
  variation model. {P}art {I}: {M}odelling and theory.
\newblock {\em J. Math. Imaging Vision}, 59(3):498--514, 2017.

\bibitem{HoK22}
G.~Holler and K.~Kunisch.
\newblock Learning nonlocal regularization operators.
\newblock {\em Math. Control Relat. Fields}, 12(1):81--114, 2022.

\bibitem{HKB18}
G.~Holler, K.~Kunisch, and R.~C. Barnard.
\newblock A bilevel approach for parameter learning in inverse problems.
\newblock {\em Inverse Problems}, 34(11):115012, 28, 2018.

\bibitem{KRZ22}
C.~Kreisbeck, A.~Ritorto, and E.~Zappale.
\newblock Cartesian convexity as the key notion in the variational existence
  theory for nonlocal supremal functionals.
\newblock {\em Nonlinear Anal.}, 225:Paper No. 113111, 2022.

\bibitem{KrZ20}
C.~Kreisbeck and E.~Zappale.
\newblock Lower semicontinuity and relaxation of nonlocal
  {$L^\infty$}-functionals.
\newblock {\em Calc. Var. Partial Differential Equations}, 59(4):Paper No. 138,
  36, 2020.

\bibitem{LiS19}
P.~Liu and C.-B. Sch\"{o}nlieb.
\newblock Learning optimal orders of the underlying {E}uclidean norm in total
  variation image denoising.
\newblock {\em Preprint, arXiv:1903.11953}, 2019.

\bibitem{Mun09}
J.~Mu\~{n}oz.
\newblock Characterisation of the weak lower semicontinuity for a type of
  nonlocal integral functional: the {$n$}-dimensional scalar case.
\newblock {\em J. Math. Anal. Appl.}, 360(2):495--502, 2009.

\bibitem{NOS15}
R.~H. Nochetto, E.~Ot\'{a}rola, and A.~J. Salgado.
\newblock A {PDE} approach to fractional diffusion in general domains: a priori
  error analysis.
\newblock {\em Found. Comput. Math.}, 15(3):733--791, 2015.

\bibitem{PPR21}
V.~Pagliari, K.~Papafitsoros, B.~Rai\c{t}\u{a}, and A.~Vikelis.
\newblock Bilevel training schemes in imaging for total-variation-type
  functionals with convex integrands.
\newblock {\em Preprint, arXiv:2112.10682}, 2021.

\bibitem{Ped97}
P.~Pedregal.
\newblock {\em Parametrized measures and variational principles}, volume~30 of
  {\em Progress in Nonlinear Differential Equations and their Applications}.
\newblock Birkh\"{a}user Verlag, Basel, 1997.

\bibitem{Ped16}
P.~Pedregal.
\newblock Weak lower semicontinuity and relaxation for a class of non-local
  functionals.
\newblock {\em Rev. Mat. Complut.}, 29(3):485--495, 2016.

\bibitem{PCBC10}
T.~Pock, D.~Cremers, H.~Bischof, and A.~Chambolle.
\newblock Global solutions of variational models with convex regularization.
\newblock {\em SIAM J. Imaging Sci.}, 3(4):1122--1145, 2010.

\bibitem{Pon04}
A.~C. Ponce.
\newblock A new approach to {S}obolev spaces and connections to
  {$\Gamma$}-convergence.
\newblock {\em Calc. Var. Partial Differential Equations}, 19(3):229--255,
  2004.

\bibitem{ROF92}
L.~I. Rudin, S.~Osher, and E.~Fatemi.
\newblock Nonlinear total variation based noise removal algorithms.
\newblock {\em Phys. D}, 60(1-4):259--268, 1992.
\newblock Experimental mathematics: computational issues in nonlinear science
  (Los Alamos, NM, 1991).

\bibitem{TaLiAdFr07}
M.~F. Tappen, C.~Liu, E.~H. Adelson, and W.~T. Freeman.
\newblock Learning gaussian conditional random fields for low-level vision.
\newblock In {\em 2007 IEEE Conference on Computer Vision and Pattern
  Recognition}, pages 1--8, 2007.

\end{thebibliography}
%

\end{document}